\documentclass[12pt]{amsart}
\usepackage[margin=1in]{geometry}
\usepackage[english]{babel}
\usepackage{graphicx}
\usepackage{amsmath,amssymb,amsthm}
\usepackage{yfonts}
\usepackage{color}
\usepackage[unicode]{hyperref}
\hypersetup{colorlinks=true, citecolor=blue, linkcolor=blue, urlcolor=blue, pdfstartview=FitH, pdfauthor=Bryce Kerr and Oleksiy Klurman and Jesse Thorner, pdftitle=Zeros of L-functions and large partial sums of Dirichlet coefficients}

\usepackage{amssymb,amsfonts,amsmath}
\usepackage{color}
\usepackage{enumerate}
\usepackage{mathrsfs}
\usepackage[capitalise]{cleveref}
\usepackage{constants}

\newtheorem{theorem}{Theorem}[section]
\newtheorem{corollary}[theorem]{Corollary}

\newtheorem{proposition}[theorem]{Proposition}
\newtheorem{lemma}[theorem]{Lemma}

\theoremstyle{definition}

\theoremstyle{remark}
\newtheorem*{remark}{Remark}

\theoremstyle{remark}

\numberwithin{equation}{section}

\crefname{figure}{Figure}{Figures}
\theoremstyle{plain}
\newtheorem*{theorem*}{Theorem}
\crefname{theorem}{Theorem}{Theorems}
\crefname{corollary}{Corollary}{Corollaries}
\newtheorem*{corollary*}{Corollary}
\crefname{lemma}{Lemma}{Lemmata}
\crefname{proposition}{Proposition}{Propositions}
\crefname{conjecture}{Conjecture}{Conjectures}
\newtheorem*{conjecture*}{Conjecture}
\crefname{conjecture*}{Conjecture}{Conjectures}
\crefname{definition}{Definition}{Definitions}
\crefname{hypothesis}{Hypothesis}{Hypotheses}

\renewcommand{\hat}{\widehat}
\newcommand{\Z}{\mathbb{Z}}
\newcommand{\R}{\mathbb{R}}
\newcommand{\Q}{\mathbb{Q}}

\newcommand{\CC}{\mathfrak{C}}

\newcommand{\re}{\textup{Re}}
\newcommand{\im}{\textup{Im}}

\newcommand{\GL}{\mathrm{GL}}

\renewcommand{\epsilon}{\varepsilon}

\renewcommand{\pmod}[1]{\, (\mathrm{mod} {\, #1})}

\renewcommand{\Re}{\mathrm{Re}}

\newcommand{\SL}{ {\rm SL}}

\newconstantfamily{abcon}{symbol=c}

\renewcommand{\Re}{\mathrm{Re}}

\newcounter{equi1}

\newcommand{\bea}{\begin{eqnarray*}}
\newcommand{\eea}{\end{eqnarray*}}
\newcommand{\beq}{\begin{equation}}
\newcommand{\eeq}{\end{equation}}
\newcommand{\begsta}{\begin{statements}}
\def\endsta{\end{statements}}
\newcommand{\begaeq}{\begin{aequivalenz}}
\def\endaeq{\end{aequivalenz}}
\begin{document}

\title[Zeros of $L$-functions and large partial sums of Dirichlet coefficients]{Zeros of $L$-functions and large partial sums of Dirichlet coefficients}
\author{Bryce Kerr}
\author{Oleksiy Klurman}
\author{Jesse Thorner}

\begin{abstract}
Let $L(s,\pi)=\sum_{n=1}^{\infty}\lambda_{\pi}(n)n^{-s}$ be an $L$-function that satisfies a weak form of the generalized Ramanujan conjecture.  We prove that large partial sums of $\lambda_{\pi}(n)$ strongly repel the low-lying zeros of $L(s,\pi)$ away from the critical line.  Our results extend and quantitatively improve preceding work of Granville and Soundararajan.
\end{abstract}
\maketitle

\section{Introduction and statement of the main results}
\label{sec:intro}

In \cite{GS_JAMS,GS_JEMS}, Granville and Soundararajan proved striking connections between the size of character sums and the distribution of zeros of the associated Dirichlet $L$-functions.  To describe their results, let $\chi\pmod{q}$ be a primitive Dirichlet character, and let
\[
S(x,\chi)=\sum_{n\leq x}\chi(n).
\]
It is conjectured that if $\epsilon>0$ and $x>q^{\epsilon}$, then $S(x,\chi)=o_{\epsilon}(x)$.  Under the generalized Riemann hypothesis (GRH) for $L(s,\chi)$, Granville and Soundararajan \cite{GS_JAMS} achieved this bound as $(\log x)/\log\log q\to\infty$.  They unconditionally proved that this is the best range possible.  In \cite[Theorems 1.3 and 1.4]{GS_JEMS}, Granville and Soundararajan further develop the relationship between zeros of $L(s,\chi)$ and the size of $S(x,\chi)$.  Let $x\in[\exp(\sqrt{\log q}),\sqrt{q}]$ and $N\in[1,(\log x)^{1/100}]$.  They prove that there exist constants $\Cl[abcon]{GS_JEMS_1},\Cl[abcon]{GS_JEMS_2},\Cl[abcon]{GS_JEMS_3}>$ such that if $|S(x,\chi)|=x/N$, then there exists $\phi\in[-\Cr{GS_JEMS_1}N,\Cr{GS_JEMS_1}N]$ such that\footnote{All counts of and sums over zeros of $L$-functions include the multiplicity of the zeros.}
\begin{equation}
\label{eqn:GS}
\#\Big\{\rho\colon L(\rho,\chi)=0,~|1+i\phi-\rho|<\frac{M}{\log x}\cdot \frac{\log q}{\log x}\Big\}\geq \frac{M}{400},\qquad M\in\Big[\Cr{GS_JEMS_2}N^6,\frac{1}{2}\log x\Big].
\end{equation}
If $\chi$ has order $k\geq 2$, then we may take $\phi=0$ when $M\geq (\Cr{GS_JEMS_3}N)^{2k^2}$.

If $k=2$, $\epsilon\in((\log q)^{-1/3},\frac{1}{2}]$, and $S(q^{\epsilon},\chi)$ is large, then \eqref{eqn:GS} with $M=(\epsilon^2/4)\log q$ implies that a positive proportion ($\gg\epsilon^2$) of the $O(\log q)$ nontrivial zeros $\beta+i\gamma$ of $L(s,\chi)$ with $\beta\in(0,1)$ and $|\gamma|\leq\frac{1}{4}$ satisfy $\beta\geq\frac{3}{4}$.  Such zeros violate GRH, the assertion that $\beta=\frac{1}{2}$ always.  As a corollary of our main result, \cref{thm:main}, we establish an improvement.

\begin{corollary}
\label{cor:main}
Let $\delta\in(0,\frac{1}{2}]$.  There exist constants $\Cl[abcon]{cor1.1_0},\Cl[abcon]{cor1.1_1},\Cl[abcon]{cor1.1_2},\Cl[abcon]{cor1.1_3},\Cl[abcon]{cor1.1_4}>0$ (depending on $\delta$) such that the following is true.  Let $q\geq \Cr{cor1.1_0}$, and let $\chi\pmod{q}$ be a primitive Dirichlet character of order $k\geq 2$.  Let $x\in[\exp((\log q)^{49/50}),\sqrt{q}]$, $N\in[1,(\log x)^{1/100}]$, and $L\in[\Cr{cor1.1_1}N^{10}(\frac{\log q}{\log x})^{1-2\delta},\Cr{cor1.1_2}\log x]$.  If $|S(x,\chi)|=x/N$, then there exists $\phi\in[-\Cr{cor1.1_3}N,\Cr{cor1.1_3}N]$ such that
\begin{equation}
\label{eqn:GS_region}
\#\Big\{\rho\colon L(\rho,\chi)=0,~|1+i\phi-\rho|\leq \frac{L}{\log x} \Big(\frac{\log q}{\log x}\Big)^{\delta}\Big\}\geq \Cr{cor1.1_4} L.
\end{equation}
One may take $\phi=0$ when $L\geq (\Cr{GS_JEMS_3} N)^{2k^2}$.
\end{corollary}

Let $\delta\in(0,\frac{1}{2}]$, and let $\chi\pmod{q}$ be a primitive Dirichlet character of order $k\geq 2$.  Let $\epsilon\in[(\log q)^{-1/50},\frac{1}{2}]$, $x=q^{\epsilon}$, and $N\in[1,(\log x)^{1/100}]$.  If
\[
|S(x,\chi)| = x/N,\qquad M \in [\Cr{cor1.1_1}N^{10},\Cr{cor1.1_2}\epsilon^{1-2\delta}\log x],\qquad L = \frac{M}{\epsilon^{1-2\delta}},
\]
then the bound \eqref{eqn:GS_region} in \cref{cor:main} becomes
\begin{equation}
\label{eqn:renormalized}
\#\Big\{\rho\colon L(\rho,\chi)=0,~|1+i\phi-\rho|<\epsilon^{\delta}\frac{M}{\log x}\cdot\frac{\log q}{\log x}\Big\}\geq \Cr{cor1.1_4} \frac{M}{\epsilon^{1-2\delta}}.
\end{equation}
This ensures that many more zeros of $L(s,\chi)$ must lie in a circle centered at $1+i\phi$ with a smaller radius than in \eqref{eqn:GS}.  Choosing $M = \min\{\frac{1}{4},\Cr{cor1.1_2}\}\epsilon^{2-\delta}\log q$, we conclude that if $S(q^{\epsilon},\chi)$ is large, then a much larger positive proportion ($\gg_{\delta}\epsilon^{1+\delta}$) of the $O(\log q)$ nontrivial zeros $\beta+i\gamma$ of $L(s,\chi)$ with $0<\beta<1$ and $|\gamma-\phi|\leq\frac{1}{4}$ satisfy $\beta\geq\frac{3}{4}$ (with $\phi=0$ when $k=2$).  This confirms a prediction of Granville and Soundararajan \cite[p. 2]{GS_JEMS}.

To describe the sharpness of the bound in \eqref{eqn:GS_region}, we recall a result of Linnik:  If $t\in\R$ and $1/\log(q(|t|+3))\leq r\leq 1$, then $\#\{\rho\colon L(\rho,\chi)=0,~|\rho-(1+it)|\leq r \}\ll r\log(q(|t|+3))$. Let $\delta,\epsilon\in(0,1]$, $x=q^{\epsilon}$, $L\geq 1$, and $|\phi|\ll q$. Applying Linnik's bound with 
\begin{align*}
r=\frac{L}{\log x} \Big(\frac{\log q}{\log x}\Big)^{\delta}, \qquad t=\phi,
\end{align*}
we obtain the unconditional upper bound
\begin{align*}
\#\Big\{\rho\colon L(\rho,\chi)=0,~|1+i\phi-\rho|<\frac{L}{\log x} \Big(\frac{\log q}{\log x}\Big)^{\delta}\Big\}\ll  \frac{L}{\epsilon^{1+\delta}}.
\end{align*}
This matches the lower bound from \eqref{eqn:GS_region} up to a factor of $(1/\epsilon)^{1+\delta}$, the reciprocal of our new and improved proportion of nontrivial zeros repelled from the critical line.

Our work extends to other $L$-functions.  For example, let $k\geq 2$ be even, and let
\[
f(z)=\sum_{n=1}^{\infty}\lambda_f(n)n^{\frac{k-1}{2}}e^{2\pi i n z}
\]
be a holomorphic cusp form on $\SL_2(\Z)$ of weight is $k$ with Hecke eigenvalues $\lambda_f(n)$.  Suppose that $f$ is an eigenfunction of each Hecke operator and that $f$ is normalized so that $\lambda_f(1)=1$.  It follows that if $\tau(n)$ is the divisor function, then $|\lambda_f(n)|\leq \tau(n)$.  Define
\[
S(x,f)=\sum_{n\leq x}\lambda_f(n),\qquad L(s,f) = \sum_{n=1}^{\infty}\frac{\lambda_f(n)}{n^s}.
\]
The $L$-function $L(s,f)$ converges absolutely for $\re(s)>1$ and satisfies a functional equation relating $L(s,f)$ to $L(1-s,f)$.  Lamzouri \cite{Lamzouri} proved that under GRH for $L(s,f)$, the bound $S(x,f)=o(x\log x)$ holds in the range $(\log x)/\log\log k \to\infty$.  Also, Lamzouri unconditionally proved that this is the best possible range.

Frechette, Gerbelli-Gauthier, Hamieh, and Tanabe \cite{FGGHT} proved a version of \eqref{eqn:GS} for $L(s,f)$.  In particular, if $\epsilon>(\log k)^{-1/8}$, $1\leq N\leq (\log k)^{1/200}$, $x\geq k^{\epsilon}$, and $S(x,f)\gg (x\log x)/N$, then there exists $\phi\in\R$ with $|\phi|\ll N$ such that a positive proportion ($\gg \epsilon^2$) of the $O(\log k_{f})$ nontrivial zeros $\beta+i\gamma$ of $L(s,\chi)$ with $0<\beta<1$ and $|\gamma-\phi|\leq \frac{1}{4}$ must satisfy $\beta\geq\frac{3}{4}$.  In other words, large partial sums of Hecke eigenvalues force a positive proportion of the nontrivial zeros of $L(s,f)$ with imaginary part close to $\phi$ to violate GRH.  As a corollary of our main result, \cref{thm:main}, we improve upon \cite{FGGHT} in two ways.

\begin{corollary}
\label{cor:main_2}
Let $\delta\in(0,\frac{1}{2}]$.  There exist constants $\Cl[abcon]{cor1.2_1},\Cl[abcon]{cor1.2_2},\Cl[abcon]{cor1.2_3},\Cl[abcon]{cor1.2_4}>0$ (depending on $\delta$) such that the following is true.  Let $f$ be a holomorphic cuspidal Hecke eigenform on $\mathrm{SL}_2(\Z)$ of even weight $k \geq \Cr{cor1.2_1}$.  Let $x\in[\exp((\log k)^{99/100}),\sqrt{k}]$, $N\in[1,(\log x)^{1/200}]$, and $L\in[\Cr{cor1.2_2}N^{5}(\frac{\log k}{\log x})^{1-2\delta},\Cr{cor1.2_3}\log x]$.  If $|S(x,f)|=(x\log x)/N$, then
\begin{equation*}
\#\Big\{\rho\colon L(\rho,\chi)=0,~|\rho-1|<\frac{L}{\log x} \Big(\frac{\log q}{\log x}\Big)^{\delta}\Big\}\geq \Cr{cor1.2_4} L.
\end{equation*}
\end{corollary}

\begin{remark}
Let $\delta\in(0,\frac{1}{2}]$, $\epsilon\in[(\log k)^{-1/200},\frac{1}{2}]$, $x=k^{\epsilon}$, and $L=(\epsilon^{1+\delta}/4)\log k$.  \cref{cor:main_2} implies that if $S(k^{\epsilon},f)$ is large, then a positive proportion ($\gg_{\delta}\epsilon^{1+\delta}$) of the $O(\log k)$ nontrivial zeros $\beta+i\gamma$ with $0<\beta<1$ and $|\gamma|\leq \frac{1}{4}$ satisfy $\beta\geq\frac{3}{4}$.  This proportion is much larger than what is shown in \cite{FGGHT}, with the added benefit that we may take $\phi=0$.
\end{remark}

To state the main result from which \cref{cor:main,cor:main_2} follow, we describe the class of $L$-functions considered by Soundararajan in \cite{Sound_weak}.  Let $m\geq 1$ be a fixed integer.  Let $L(s,\pi)$ be given by the Dirichlet series and Euler product
\begin{equation}
	\label{eqn:Dirichlet_Euler}
	L(s,\pi)=\sum_{n=1}^{\infty}\frac{\lambda_{\pi}(n)}{n^s}=\prod_p\prod_{j=1}^m\Big(1-\frac{\alpha_{j,\pi}(p)}{p^s}\Big)^{-1},
\end{equation}
both of which converge absolutely in the region $\re(s)>1$.  We write
\begin{equation}
	\label{eqn:gamma_factor}
	L(s,\pi_{\infty})=q_{\pi}^{s/2}\prod_{j=1}^m \Gamma\Big(\frac{s+\mu_{\pi}(j)}{2}\Big),
\end{equation}
where $q_{\pi}$ denotes the conductor and the $\mu_{\pi}(j)$ are complex numbers.  We assume that there exists $\theta_{\pi}\in[0,1-\frac{1}{m}]$ such that for all primes $p$ and all $1\leq j\leq m$, there holds
\begin{equation}
	\label{eqn:weak_selberg}
|\alpha_{j,\pi}(p)|\leq p^{\theta_{\pi}},\qquad \re(\mu_{\pi}(j))\geq -\theta_{\pi}.
\end{equation}
The generalized Ramanujan conjecture (GRC) for $L(s,\pi)$ is the assertion that $\theta_{\pi}=0$.

The completed $L$-function $L(s,\pi)L(s,\pi_{\infty})$ is an entire function of order 1.\footnote{In this paper, $L(s,\pi)$ has no poles. We could easily modify the proofs to allow for a pole at $s=1$.}  Moreover, there exists a complex number $W(\pi)$ of modulus $1$ such that
\begin{equation}
	\label{eqn:functional_equation}
	L(s,\pi)L(s,\pi_{\infty})=W(\pi)\overline{L(1-\overline{s},\pi)L(1-\overline{s},\pi_{\infty})}.
\end{equation}
We define the analytic conductor $\CC$ of $L(s,\pi)$ by
\begin{equation}
	\label{eqn:analytic_conductor}
	\CC(it)=q_{\pi}\prod_{j=1}^m(|it+\mu_{\pi}(j)|+3),\qquad \CC=\CC(0).
\end{equation}

We require a hypothesis towards an averaged bound for the numbers $\alpha_{j,\pi}(p)$ in \eqref{eqn:weak_selberg}.  Let $\Lambda(n)$ be the von Mangoldt function.  Define the numbers $a_{\pi}(n)$ by
\begin{equation}
	\label{eqn:logarithmic_derivative}
	-\frac{L'}{L}(s,\pi)=\sum_{n=1}^{\infty}\frac{a_{\pi}(n)\Lambda(n)}{n^s}.
\end{equation}
Using \eqref{eqn:Dirichlet_Euler} in the region $\re(s)>1$, we see that
\[
a_{\pi}(n)=\begin{cases}
\sum_{j=1}^m \alpha_{j,\pi}(p)^k&\mbox{if there is a prime $p$ and an integer $k\geq 1$ such that $n=p^k$,}\\
0&\mbox{otherwise.}
\end{cases}
\]
We assume that there exist integers $\kappa,A_0\geq 1$ such that
\begin{equation}
	\label{eqn:weak_ramanujan}
	\textup{if $x\geq 1$, then }\sum_{x<n\leq ex}\frac{|a_{\pi}(n)|^2}{n}\Lambda(n)\leq \kappa^2+\frac{A_0}{\log ex}.
\end{equation}

Given $m\geq 1$, we denote by $\mathfrak{S}(m)$ the set of entire $L$-functions satisfying \eqref{eqn:Dirichlet_Euler}--\eqref{eqn:weak_ramanujan}.  For $x\geq 2$ and an $L$-function $L(s,\pi)\in\mathfrak{S}(m)$, we have the ``trivial'' bound (see \cref{lem:triv_bound_holds})
\begin{equation}
\label{eqn:trivial_bound_pi}
S(x,\pi)\ll_{\kappa,A_0,m}x(\log x)^{\kappa-1},\qquad S(x,\pi):=\sum_{n\leq x}\lambda_{\pi}(n).
\end{equation}
By modifying the ideas of Lamzouri in \cite{Lamzouri}, one can produce for each $m\geq 1$ an infinitude of $L(s,\pi)\in\mathfrak{S}(m)$ with $\kappa=m$ and $A_0\ll m^2$ such that if $\log x\asymp_{A_0,m} \log\log C(\pi)$, then $S(x,\pi)\asymp_{A_0,m}x(\log x)^{\kappa-1}$.\footnote{For each suitably large integer $k\equiv 0\pmod{4}$ and each integer $m\geq 2$, one can extend Lamzouri's work (with minor modifications) to the $m$-th symmetric power $L$-functions of the holomorphic cusp forms of weight $k$ on $\SL_2(\Z)$.}  On the other hand, GRH for $L(s,\pi)$ implies that if $\epsilon>0$ and $x\geq \CC^{\epsilon}$, then $S(x,\pi)=o_{\epsilon,\kappa,A_0,m}(x(\log x)^{\kappa-1})$.  Define
\begin{equation}
\label{eqn:ell_kappa_def}
\ell_{\kappa}:=\min\{\kappa(1-\sqrt{2/\pi}),1\}+1/1000.
\end{equation}
We prove the following theorem.
\begin{theorem}
\label{thm:main}
Let $\delta\in(0,\frac{1}{2}]$.  Let $m\geq 1$, and let $L(s,\pi)\in\mathfrak{S}(m)$ with analytic conductor $\CC$.  Let $\kappa$ and $A_0$ be as in \eqref{eqn:weak_ramanujan}.  There exist constants $\Cl[abcon]{thm1.3_1},\Cl[abcon]{thm1.3_2},\Cl[abcon]{thm1.3_3},\Cl[abcon]{thm1.3_4},\Cl[abcon]{thm1.3_5},\Cl[abcon]{thm1.3_6}>0$ (depending on $\delta$, $\kappa$, $m$, and $A_0$) such that the following is true.  Let $\CC\geq \Cr{thm1.3_1}$, $x\in[\exp((\log\CC)^{1-1/(50m)}),\sqrt{\CC}]$, $N\in[1,(\log x)^{1/(100m)}]$, and $L\in[\Cr{thm1.3_2}N^{2/\ell_{\kappa}}(\frac{\log \CC}{\log x})^{1-2\delta},\Cr{thm1.3_3}\log x]$. If $|S(x,\pi)|=x(\log x)^{\kappa-1}/N$, then there exists $\phi\in\R$ such that $|\phi|\leq\Cr{thm1.3_4}N$ and
\begin{equation*}
\#\Big\{\rho\colon L(\rho,\pi)=0,~|1+i\phi-\rho|<\frac{L}{\log x} \Big(\frac{\log \CC}{\log x}\Big)^{\delta}\Big\}\geq \Cr{thm1.3_5} L.
\end{equation*}
If $a_{\pi}(n)\in\R$ for all $n\geq 1$, then one may take $\phi=0$ when $L\geq \Cr{thm1.3_6}N^{6/\kappa}$.
\end{theorem}

\begin{remark}
Let $\epsilon\in[(\log \CC)^{-1/(50m)},\frac{1}{2}]$.  \cref{thm:main} (with $L= (\epsilon^{1+\delta}/4)\log \CC$) implies that if 
$S(\CC^{\epsilon},\pi)$ is large, then a positive proportion ($\gg_{\delta,\kappa,m,A_0}\epsilon^{1+\delta}$) of the $O(\log \CC)$ nontrivial zeros $\beta+i\gamma$ of $L(s,\chi)$ with $0<\beta<1$ and $|\gamma-\phi|\leq \frac{1}{4}$ must satisfy $\beta\geq\frac{3}{4}$.  Since GRH for $L(s,\pi)$ asserts that $\beta=\frac{1}{2}$ for each such zero, we see that large partial sums of $\lambda_{\pi}(n)$ strongly repel the nontrivial zeros away from the critical line.  If the numbers $a_{\pi}(n)$ are real, then we may take $\phi=0$.
\end{remark}

\begin{remark}
The most general $L$-functions satisfying the hypotheses of \cref{thm:main} are described in \cite[Example 3]{Sound_weak}.  Let $\mathbb{A}_{\Q}$ be the adele ring of $\Q$, let $m\geq 1$, and let $\mathfrak{F}_m$ be the set of nontrivial cuspidal automorphic representations of $\GL_m(\mathbb{A}_{\Q})$ with unitary central character.  If $\pi\in\mathfrak{F}_m$ satisfies GRC, then $L(s,\pi)\in\mathfrak{S}(m)$ with $\kappa\leq m$ and $A_0\ll m^2$.  \cref{thm:main} now applies to $L(s,\pi)$, with $\phi=0$ when $\pi$ is self-dual.  If $\pi_1\in\mathfrak{F}_{m_1}$, $\pi_2\in\mathfrak{F}_{m_2}$, $\pi_2$ is not the dual of $\pi_1$, and $\pi_1$ satisfies GRC, then the Rankin--Selberg $L$-function $L(s,\pi_1\times\pi_2)$ lies in $\mathfrak{S}(m_1 m_2)$ with $\kappa\leq m$ and $A_0\ll_{m_1,\pi_2}1$.  One may take $\phi=0$ when $\pi_1$ and $\pi_2$ are self-dual.  This discussion subsumes \cref{cor:main,cor:main_2}.
\end{remark}

Our proof of \cref{thm:main} has two new components.  First, in \cref{sec:mean_values}, we modify the work in Granville, Harper, and Soundararajan \cite{GHS} (see also Mangerel \cite{Mang}) on mean values of multiplicative functions, leading to a broadly applicable version of Hal{\'a}sz's theorem that accommodates the class $\mathfrak{S}(m)$ of $L$-functions to which \cref{thm:main} applies.  The possibility of such an extension was suggested in \cite{GHS}.  This extension permits us to establish a version of \cite[Theorem 1.3]{GS_JEMS} for higher degree $L$-functions. We remark that the degree $2$ case has been worked out in \cite{FGGHT} under a stronger assumption on the coefficients, requiring pointwise bounds on $a_{\pi}(n)$ in place of the averaged statement given by \eqref{eqn:weak_ramanujan}.  Perhaps the most novel part in the first component of our proof is the verification that if the coefficients $a_f(n)$ are real for all $n\in\R$, then one may take $\phi=0$ in \cref{thm:main}, the key component being \cref{lem:twist}.

Second, we relate the zeros of $L$-functions to our work on mean values of multiplicative functions using an approach different from that of Granville and Soundararajan in \cite{GS_JEMS}.  The novelty here lies in the derivation of a new ``hybrid Euler--Hadamard product'' formula for $L(s,\pi)$ in \cref{sec:explicit}.  Pioneered by Gonek, Hughes, and Keating \cite{GonekHughesKeating} (whose work is based on that of Bombieri and Hejhal \cite{BombieriHejhal}), the hybrid Euler--Hadamard product was originally developed to study moments of the Riemann zeta function $\zeta(s)$ on the critical line $\re(s)=\frac{1}{2}$ by incorporating information from both the zeros of $\zeta(s)$ in the Hadamard product formulation and the primes in the Euler product formulation of $\zeta(s)$.  One can choose whether Euler--Hadamard product depends more on the contribution from the primes or the zeros by adjusting a single parameter.  In order to study the moments of $\zeta(s)$ on $\re(s)=\frac{1}{2}$, Gonek, Hughes, and Keating favored the contribution from the primes and then applied the Riemann Hypothesis to handle the zeros.  In several papers citing \cite{GonekHughesKeating} (e.g., \cite{BuiKeating,FarmerGonekHughes}), this approach is applied to study $L$-functions (either individually or in families) on the line $\re(s)=\frac{1}{2}$, where GRH handles the contribution from the zeros.  Here, we establish a hybrid Euler--Hadamard product formula for $L$-functions in $\mathfrak{S}(m)$, favoring the zeros instead of the primes, to study $L$-functions near the line $\re(s)=1$ without assuming GRH.  All of these facets appear to be new in our application of the ideas underlying the Euler--Hadamard formulation of $L$-functions.

Throughout this paper, all implied constants in our $O(\cdot)$, $\ll$, and $\gg$ notation depend on at most on the numbers $\kappa$, $A_0$, $m$, and $\delta$ in \cref{thm:main}.  All implied constants are effectively computable.  Each constant in the sequence $c_1,c_2,c_3,\ldots$ is effectively computable.

\subsection*{Acknowledgements}
B.K. and O.K. are greatful to the MPIM (Bonn) for providing excellent working conditions and support.  J.T. is partially supported by the National Science Foundation (DMS-2401311) and the Simons Foundation (MP-TSM-00002484).

\section{Mean values of multiplicative functions}
\label{sec:mean_values}

Let $f\colon\Z\to\mathbb{C}$ be a multiplicative function.  We write
\[
F(s)=\sum_{n=1}^{\infty}\frac{f(n)}{n^s},
\]
which we assume converges absolutely for $\re(s)>1$.  Let $\Lambda(n)$ be the von Mangoldt function, and as in \eqref{eqn:logarithmic_derivative}, we define the numbers $a_f(n)$ by
\[
-\frac{F'}{F}(s)=\sum_{n=1}^{\infty}\frac{a_{f}(n)\Lambda(n)}{n^s}.
\]

Part of our work is based on results from  \cite{GHS}, where the authors develop the theory of multiplicative functions, restricted to the class $\mathcal{C}(\kappa)$ of multiplicative functions $f$ such that if $n\geq 1$, then $|a_f(n)|\leq\kappa$.  When $f(n)=\lambda_{\pi}(n)$ with $\pi$ a cuspidal automorphic representation on $\mathrm{GL}(m)$, the condition $f\in\mathcal{C}(m)$ follows from GRC.  In \cite{GHS}, it is mentioned that one could establish corresponding results with the weaker assumption
that there exists some $A_0\geq 1$ such that
\begin{equation}
\label{eqn:weak_ramanujan_2}
\sum_{x<n\leq ex}\frac{|a_f(n)|^2 \Lambda(n)}{n}\leq \kappa^2+\frac{A_0}{\log(ex)}\qquad \textup{for all $x\geq 1$}
\end{equation}
(cf. \eqref{eqn:weak_ramanujan}).  We denote the class of multiplicative functions $f$ which satisfy \eqref{eqn:weak_ramanujan_2} by $\mathcal{C}_w(\kappa)$.

\subsection{Preliminary estimates}
We collect some estimates from \cite{Sound_weak} for $f\in\mathcal{C}_w(\kappa)$.

\begin{lemma}[{\cite[Lemma 3.2]{Sound_weak}}]
\label{lem:sound_weak_lemma_3.2}
	Let $f\in\mathcal{C}_w(\kappa)$.  If $\sigma\in(1,2]$, then $|F(\sigma+it)|\ll_{\kappa,A_0} (\sigma-1)^{-\kappa}$.
\end{lemma}

\begin{lemma}[{\cite[Lemma 3.3]{Sound_weak}}]
\label{lem:sound_weak_lemma_3.3}
	Let $f\in\mathcal{C}_w(\kappa)$.  If $x\geq 2$ and $1\leq y\leq x$, then
	\[
	\sum_{x<n\leq x+y}|f(n)|\ll_{\kappa,A_0} (yx)^{1/2}(\log x)^{\kappa^2/2}.
	\]
\end{lemma}

\begin{lemma}[{\cite[Equations 3.3 and 3.4]{Sound_weak}}]
\label{lem:sigma_near_1_bounds}
	If $f\in\mathcal{C}_w(\kappa)$, then for any $\sigma\in(1,2]$, we have
	\[
	\sum_{n=2}^{\infty}\frac{|a_f(n)|^2\Lambda(n)}{n^{\sigma}\log n}\leq \kappa^2\log\frac{1}{\sigma-1}+O_{\kappa,A_0}(1),\quad \sum_{n=2}^{\infty}\frac{|a_f(n)|\Lambda(n)}{n^{\sigma}\log n}\leq \kappa\log\frac{1}{\sigma-1}+O_{\kappa,A_0}(1).
	\]
\end{lemma}

\cref{lem:sigma_near_1_bounds} furnishes a strong bound for $F(s)$ near the line $\re(s)=1$.

\begin{lemma}
\label{lem:lemma4.1GHS}
	Let $f\in\mathcal{C}_w(\kappa)$.  Select a real number $t_1$ with $|t_1|\leq (\log x)^{\kappa}$ that maximizes $|F(1+1/\log x+it_1)|$.  If $|t|\leq(\log x)^{\kappa}$, then
	\[
	|F(1+1/\log x+it)|\ll_{\kappa,A_0}(\log x)^{\kappa\sqrt{2/\pi}}\Big(\frac{\log x}{1+|t-t_1|\log x}+(\log\log x)^2\Big)^{\kappa(1-\sqrt{2/\pi})}.
	\]
\end{lemma}
\begin{proof}
This is similar to the proof of \cite[Lemma 4.1]{GHS}, except in our use of \eqref{eqn:weak_ramanujan_2}.  If $|t-t_1|\geq 1/\log x$, then the stated estimate is immediate from \cref{lem:sound_weak_lemma_3.2} with $\sigma=1+1/\log x$.  We now suppose that $(\log x)^{\kappa}\gg|t-t_1|\geq 1/\log x$.  By the maximality of $t_1$, we have $|F(1+1/\log x+it)|\leq |F(1+1/\log x+it_1)|$.  In light of this, we set $\tau=|t_1-t|/2$.
	
	From the inequality $\cos^2(\tau \log n)\leq |\cos(\tau\log n)|$ combined with the Cauchy--Schwarz inequality and \eqref{eqn:weak_ramanujan_2}, we find that
\begin{align*}
\log|F(1+\tfrac{1}{\log x}+it)|&\leq \frac{1}{2}(\log|F(1+\tfrac{1}{\log x}+it_1)|+\log|F(1+\tfrac{1}{\log x}+it)|)\\
&= \re\Big(\sum_{n\geq 2}\frac{a_f(n) \Lambda(n)}{n^{1+1/\log x}\log n}\cdot n^{i(t_1+t)/2}\cdot\frac{n^{i\tau}+n^{-i\tau}}{2}\Big)\\
&\leq \sum_{n\geq 2}\frac{|a_f(n)| \Lambda(n)}{n^{1+1/\log x}\log n}|\cos(\tau \log n)|\\
&\leq \Big(\sum_{n\geq 2}\frac{\cos^2(\tau\log n)\Lambda(n)}{n^{1+1/\log x}\log n}\Big)^{\frac{1}{2}}\Big(\sum_{n\geq 2}\frac{|a_f(n)|^2 \Lambda(n)}{n^{1+1/\log x}\log n}\Big)^{\frac{1}{2}}\\
&\leq \Big(\sum_{p\leq  x}\frac{|\cos(\tau\log p)|}{p}+O(1)\Big)^{\frac{1}{2}}\Big(\sum_{n\geq 2}\frac{|a_f(n)|^2 \Lambda(n)}{n^{1+1/\log x}\log n}\Big)^{\frac{1}{2}}.
\end{align*}
We apply \cref{lem:sigma_near_1_bounds} with $\sigma=1+1/\log x$ and use partial summation to conclude that
\[
\sum_{n\geq 2}\frac{|a_f(n)|^2 \Lambda(n)}{n^{1+1/\log x}
\log n}\leq \kappa^2\log \log x+ O_{\kappa,A_0}(1).
\]
By \cite[Lemma 2.3]{GS_decay}, there exists a constant $\Cr{Lem2.4_large}>0$ such that if
\[
y=\max\{\exp(\Cl[abcon]{Lem2.4_large}(\log \log x)^2),\exp(1/|\tau|)\},
\]
then
\begin{equation}
\label{eqn:cosine_asymptotic}
\sum_{n\leq x}\frac{\cos^2(\tau\log n)\Lambda(n)}{n\log n}\leq \sum_{p\leq x}\frac{|\cos(\tau\log p)|}{p}+O(1)\leq \frac{2}{\pi}\log\frac{\log x}{\log y}+\log \log y+O(1).
\end{equation}
Combining the above estimates yields the lemma.
\end{proof}

\cref{lem:lemma4.1GHS} leads us to the following Lipschitz estimate.

\begin{lemma}
\label{lem:lipshitz}
Let $f\in\mathcal{C}_w(\kappa)$, and let $t_1$ be as in \cref{lem:lemma4.1GHS}.  If $1\leq \omega\leq x^{1/3}$, then
\begin{multline*}
\Big|\frac{1}{x}\sum_{n\leq x}\frac{f(n)}{n^{it_1}}-\frac{\omega}{x}\sum_{n\leq x/\omega}\frac{f(n)}{n^{it_1}}\Big|\\
\ll \Big(\frac{\log \omega+(\log \log x)^2}{\log x}\Big)^{\min \{1,\kappa(1-\sqrt{2/\pi})\}}(\log x)^{\kappa-1} \log\Big(\frac{\log x}{1+\log\omega}\Big).
\end{multline*}
\end{lemma}
\begin{proof}
First, let $g\in\mathcal{C}_w(\kappa)$, and let $t_0\in[-(\log x)^{\kappa},(\log x)^{\kappa}]$ be chosen so that the absolute value of the Dirichlet series for $g$ at $s=1+1/\log x + it$ is maximized when $t=t_0$.  We will demonstrate that if $1\leq \omega\leq x^{\frac{1}{3}}$, then
\begin{equation}
\label{eqn:lipschitz_prelim}
\begin{aligned}
&\Big|\frac{1}{x^{1+it_0}}\sum_{n\leq x}g(n)-\frac{1}{(x/\omega)^{1+it_0}}\sum_{n\leq x/\omega}g(n)\Big|\\
&\ll \Big(\frac{\log \omega+(\log \log x)^2}{\log x}\Big)^{\min \{1,\kappa(1-\sqrt{\frac{2}{\pi}})\}}(\log x)^{\kappa-1} \log\Big(\frac{\log x}{1+\log\omega}\Big).
\end{aligned}
\end{equation}
To achieve this, we proceed just as in the proof of \cite[Theorem 1.5]{GHS}, except that we appeal to \cref{lem:lemma4.1GHS} in place of \cite[Lemma 4.1]{GHS}.

Now, given $f\in\mathcal{C}_w(\kappa)$, let $t_1\in[-(\log x)^{\kappa},(\log x)^{\kappa}]$ be chosen so that $|F(1+1/\log x+it)|$ is maximized when $t=t_1$.  We apply \eqref{eqn:lipschitz_prelim} with $g(n) = f(n)/n^{i t_1}$.  The Dirichlet series for $g$ is then $F_{t_1}(s):=F(s+it_1)$.  By construction, the value $t_0$ in \eqref{eqn:lipschitz_prelim} is zero, and the desired result follows.
\end{proof}

\subsection{Extending Granville--Harper--Soundararajan}

For functions $f\in\mathcal{C}(\kappa)$, the following (slightly weakened) form of Hal{\'a}sz's theorem is proved in \cite[Corollary 1.2]{GHS}:  If $0<\lambda\leq\kappa$ is such that $\sum_{n\geq 1}|f(n)|n^{-1-\sigma}\ll 1/\sigma^{\lambda}$ for all $\sigma>0$, then
\[
\sum_{n\leq x}f(n)\ll_{\kappa,\lambda}(1+M)e^{-M}x(\log x)^{\lambda-1}+\frac{x}{\log x}(\log\log x)^{\kappa},
\]
where $M$ is defined by
\begin{equation}
\label{eqn:M_def}
\max_{|t|\leq(\log x)^{\kappa}}\Big|\frac{F(1+1/\log x+it)}{1+1/\log x+it}\Big|=:e^{-M}(\log x)^{\lambda}.
\end{equation}
We now record the analogue of \cite[Corollary 2.8]{GHS} for $f\in\mathcal{C}_{w}(\kappa)$.

\begin{lemma}
\label{lem:relativehalasz1}
If $f\in\mathcal{C}_w(\kappa)$ and $M$ is defined by
\[
\max_{|t|\leq (\log x)^{\kappa}}\Big|\frac{F(1+1/\log x +it)}{1+1/\log x+it}\Big|:=e^{-M}(\log x)^{\kappa},
\]
then
\[
\sum_{n\leq x}f(n)\ll_{\kappa,A_0} (1+M)e^{-M}x(\log x)^{\kappa-1}+\frac{x}{\log x}(\log\log x)^{2}.
\]
\end{lemma}
\begin{proof}[Sketch of proof]
The proofs in \cite{GHS} carry over without noteworthy change to the situation where $f\in\mathcal{C}_w(\kappa)$; this is already noted in \cite{GHS}.  Instead of applying a pointwise bound for $\lambda_f(n)$ (as one may do when $f\in\mathcal{C}(\kappa)$), one must use the Cauchy--Schwarz inequality to relate sums of $\lambda_f(n)$ to the corresponding $L^2$ estimates.  One then applies \eqref{eqn:weak_ramanujan_2}.  A prototype for such a calculation can be found in our proof of \cref{lem:lemma4.1GHS}.
\end{proof}

We now justify the assertion in \cref{sec:intro} that \eqref{eqn:trivial_bound_pi} constitutes the ``trivial bound.''

\begin{lemma}
\label{lem:triv_bound_holds}
If $f\in\mathcal{C}_w(\kappa)$ and $x\geq 2$, then $\sum_{n\leq x}f(n)\ll_{\kappa,A_0,m}x(\log x)^{\kappa-1}$.
\end{lemma}
\begin{proof}
Applying \cref{lem:sound_weak_lemma_3.2} with $\sigma=1+1/\log x$, we find that $|F(1+1/\log x+it)|\ll_{\kappa,A_0} (1+\log x)^{\kappa}$.  Therefore, the quantity $M$ in \cref{lem:relativehalasz1} is at least $-o_{\kappa,A_0}(1)$ as $x\to\infty$.  It follows that $(1+M)e^{-M}\ll_{\kappa,A_0} 1$, and the desired result now follows from \cref{lem:relativehalasz1}.
\end{proof}

Let $\phi\in\R$.  Our next lemma relates the mean values of $f(n)$ to those of $f(n)/n^{i\phi}$.
\begin{lemma}
\label{lem:factoring}
If $f\in\mathcal{C}_w(\kappa)$ and $\phi\in\mathbb{R}$, then
\begin{align*}
\Big|\sum_{n\leq x}f(n)- \frac{x^{i\phi}}{1+i\phi}\sum_{n\leq x}\frac{f(n)}{n^{i\phi}}\Big|\ll  x\frac{(\log \log x)^{\kappa}\log \log (e+|\phi|)}{\log x} \exp\Big( \sum_{n\leq x}\frac{|a_f(n)-n^{i\phi}|\Lambda(n)}{n\log n}\Big).
\end{align*}
\end{lemma}
\begin{proof}
If $h$ is a nonnegative multiplicative function and $h\in \mathcal{C}_w(\alpha),$ then by \cref{lem:relativehalasz1},
\begin{equation}
\label{halber}
\begin{aligned}
\sum_{n\leq x}h(n)&\ll \frac{x}{\log x}(1+M)\max_{|t|\leq (\log x )^{\alpha}}\Big|\frac{H(1+1/\log x +it)}{1+1/\log x+it}\Big|+\frac{x}{\log x}(\log \log x)^{\alpha}\\
& \ll \frac{x}{\log x}(\log \log x)\sum_{n\geq 1}\frac{h(n)}{n^{1+\frac{1}{\log x}}}+\frac{x}{\log x}(\log \log x)^{\alpha}\\
&\ll\frac{x}{\log x}(\log \log x)\sum_{n\leq x}\frac{h(n)}{n}+\frac{x}{\log x}(\log \log x)^{\alpha}.
\end{aligned}
\end{equation}
Define $g(n)=\sum_{d\mid n}f(d)$.  Note that $a_f(n)=a_g(n)+1$; since $|a_g(n)|^2\leq 2(|a_f(n)|^2+1)$, we have that
\[
\sum_{x<n\leq ex}\frac{|a_g(n)|^2\Lambda(n)}{n}\leq 2\sum_{x<n\leq ex}\frac{|a_f(n)|^2\Lambda(n)}{n}+2\sum_{x<n\leq ex}\frac{\Lambda(n)}{n}\leq 2(\kappa+1)+O\Big(\frac{A_0}{\log(ex)}\Big).
\]
It follows that $g\in\mathcal{C}_w(2(\kappa+1)).$  We have $\sum_{n\leq x}n^{i\phi}= \frac{x^{1+i\phi}}{1+i\phi}+O(\min\{1+|\phi|^2,x\})$ by partial summation.  Unfolding the convolution and applying \eqref{halber} for $h=|g|$, we deduce
\begin{equation}
\label{convolve}	
\begin{aligned}
&\sum_{n\leq x}f(n)n^{i\phi}\\
&=\sum_{d\leq x}g(d)d^{i\phi}\sum_{n\leq x/d}n^{i\phi}\\
&=\frac{x^{i\phi}}{1+i\phi}\sum_{d\leq x}\frac{g(d)}{d}+O\Big((1+\phi^2)\sum_{d\leq x/(1+\phi^2)}|g(d)|+x\sum_{x/(1+\phi^2)< n\leq x}\frac{|g(d)|}{d}
\Big)\\&=\frac{x^{1+i\phi}}{1+i\phi}\sum_{d\leq x}\frac{g(d)}{d}+O\Big(\frac{x \log (e+|\phi|)\log \log x}{\log x}\sum_{d\leq x}\frac{|g(d)|}{d}+\frac{x}{\log x}(\log \log x)^{\kappa+1}\Big).
\end{aligned}
\end{equation}
We note that 
\begin{align*}
\sum_{n\leq x}\frac{|g(d)|}{d}\ll \sum_{d\geq 2}\frac{|g(d)|}{d^{1+\frac{1}{\log (ex)}}}  \ll  \exp\Big( \sum_{n\geq 2}\frac{|a_g(n)|\Lambda(n)}{n^{1+\frac{1}{\log (ex)}}\log n}\Big)\ll \exp\Big( \sum_{n\leq x}\frac{|a_f(n)-1|\Lambda(n)}{n\log n}\Big).
\end{align*}
By taking $\phi =0$ in \eqref{convolve}, we find that
\[
\sum_{n\leq x}f(n)=x\sum_{d\leq x}\frac{g(d)}{d}+O\Big(\sum_{d\leq x}|g(d)|+\frac{x}{\log x}(\log \log x)^{\kappa+1}\Big).
\]
Using \eqref{halber} once again to bound the error term and substituting the result into \eqref{convolve}, we conclude that
\[
\Big|\sum_{n\leq x}f(n)n^{i\phi}- \frac{x^{i\phi}}{1+i\phi}\sum_{n\leq x}f(n)\Big|\ll x\frac{(\log \log x)^{\kappa}\log \log (e+|\phi|)}{\log x} \exp\Big( \sum_{n\leq x}\frac{|a_f(n)-1|\Lambda(n)}{n\log n}\Big).
\]
The lemma follows from the above display by replacing $f(n)$ with $f(n)/n^{i\phi}$.
\end{proof}

The next lemma provides us with a convenient alternative formulation of \cref{lem:factoring}.

\begin{lemma}
\label{lem:twist}
Let $f\in\mathcal{C}_w(\kappa)$ and $y_0\geq 1$. Define $N\geq 1$ by 
\begin{align*}
\Big|\sum_{n\leq e^{y_0}}f(n)\Big|=\frac{e^{y_0}y_0^{\kappa-1}}{N},
\end{align*}
and assume that $N\le y_0^{1/10}$.
There exists a real $\phi=\phi(y_0)$ such that $|\phi|\ll N$ and
\begin{equation}
\label{switch}
\Big|\sum_{n\leq e^{y_0}}\frac{f(n)}{n^{i\phi}}-\frac{1+i\phi}{e^{i\phi y_0}}\sum_{n\leq e^{y_0}}f(n)\Big|\ll e^{y_0}y_0^{\kappa-1-\frac{1}{5}}.
\end{equation}
If $a_f(n)\in\R$ for all $n\geq 1$, then there exists a constant $\Cl[abcon]{phi_real}=\Cr{phi_real}(\kappa,A_0,m)>0$ such that
\begin{align}
\label{eq:phifreal}
|\phi|\leq \Cr{phi_real}\frac{N^{6/\kappa}}{y_0}.
\end{align}
\end{lemma}
\begin{proof}
A stronger result follows from \cite[Lemma 3.3]{GS_JEMS} when $\kappa=1$, so we may assume that $\kappa\geq 2$. 
 Let $S$ denote the left side of~\eqref{switch}. Select $|\phi|\ll N$ so as to maximize the quantity
\begin{align}
\label{eq:gvh54+}
\max_{|t|\leq y_0^{\kappa}}\Big|\frac{F(1+1/y _0+it)}{1+1/y_0+it}\Big|=\Big|\frac{F(1+1/y _0+i\phi)}{1+1/y_0+i\phi}\Big|:=e^{-M}y^{\kappa}_0.
\end{align}
(The choice of $\phi$ will depend on $y_0$.)  Our hypothesis on $N$ implies that, $|\phi|\ll y_0^{1/10}$.  By \cref{lem:factoring} (with $x$ replaced by $e^{y_0}$), we have that
\begin{equation}
\label{eq:S11}
S\ll N\cdot e^{y_0}\frac{(\log y_0)^{\kappa}\log \log (e+|\phi|)}{y_0}\exp\Big( \sum_{n\leq e^{y_0}}|a_f(n)n^{-i\phi}-1|\frac{\Lambda(n)}{n\log n}\Big).
\end{equation}
To bound the exponential, we apply the Cauchy--Schwarz inequality to obtain
\begin{align*}
&\sum_{n\leq e^{y_0}}|a_f(n)n^{-i\phi}-1|\frac{\Lambda(n)}{n\log n}\\
&\leq \Big(\sum_{2\leq n\leq e^{y_0}}\frac{\Lambda(n)}{n\log n}\Big)^{\frac{1}{2}}\Big(\sum_{n\leq e^{y_0}}|1-a_f(n)n^{-i\phi}|^2\frac{\Lambda(n)}{n\log n}\Big)^{\frac{1}{2}}\\
&= \Big(\sum_{2\leq n\leq e^{y_0}}\frac{\Lambda(n)}{n\log n}\Big)^{\frac{1}{2}}\Big(\sum_{2\leq n\leq e^{y_0}}\frac{(|a_f(n)|^2+1)\Lambda(n)}{n\log n}-2 \sum_{2\leq n\leq e^{y_0} }\frac{\Re(a_{f}(n)n^{-i\phi})\Lambda(n)}{n\log n}\Big)^{1/2}.
\end{align*}
It follows from \eqref{eqn:weak_ramanujan_2}, dyadic decomposition, and partial summation that
\begin{equation}
\label{eqn:generalized_mertens}
\sum_{2\leq n\leq e^{y_0}}\frac{|a_f(n)|^2\Lambda(n)}{n\log n} \leq \kappa^2\log y_0 + O(1),\qquad \sum_{2\leq n\leq e^{y_0}}\frac{\Lambda(n)}{n\log n} \leq \log y_0 + O(1),
\end{equation}
so we arrive at
\begin{equation}
\label{eq:sbggjgjgjgjgjgjjggjgj}
\begin{aligned}
&\sum_{n\leq e^{y_0}}|a_f(n)n^{-i\phi}-1|\frac{\Lambda(n)}{n\log n} \\
&\le (\log{y_0}+O(1))^{1/2}\Big( (\kappa^2+1)\log{y_0}-2 \sum_{n\leq e^{y_0} }\frac{\Re(a_{f}(n)n^{-i\phi})\Lambda(n)}{n\log n}\Big) \Big)^{1/2}.
\end{aligned}
\end{equation}

We now provide a lower bound for the sums
\begin{equation*}
\sum_{n\leq e^{y_0} }\frac{\Re(a_{f}(n)n^{-i\phi})\Lambda(n)}{n\log n}.
\end{equation*}
By \cref{lem:relativehalasz1}, we have
\[
\frac{e^{y_0}y_0^{\kappa-1}}{N}\ll \Big|\sum_{n\leq e^{y_0}}f(n)\Big|\ll (1+M)e^{-M}e^{y_0}y_0^{\kappa-1},
\]
in which case
\begin{align}
\label{eq:MNlb}
M\leq \log N+2\log\log N.
\end{align}
 On the other hand, by \eqref{eq:gvh54+} and Dirichlet series manipulations, we have the bounds
\begin{equation}
\label{eq:Mabcdefgh}
\begin{aligned}
e^{-M}y_0^{\kappa}=\Big|\frac{F(1+1/y_0 +i\phi)}{1+1/y_0+i\phi}\Big|&\ll\frac{\exp(\log|F(1+1/y_0+i\phi)|)}{1+|\phi|} \\
&\ll \frac{1}{1+|\phi|}\exp\Big(\sum_{n\geq 2 }\frac{\Re(a_f(n) n^{-i\phi})\Lambda(n)}{n^{1+\frac{1}{y_0}}\log n}\Big)\\
&\asymp \frac{1}{1+|\phi|}\exp\Big(\sum_{n\leq e^{y_0} }\frac{\Re(a_f(n) n^{-i\phi})\Lambda(n)}{n\log n}\Big).
\end{aligned}
\end{equation}
We combine \eqref{eq:MNlb}, \eqref{eq:Mabcdefgh}, and the bounds $N\ll y_0^{1/10}\leq y_0^{\kappa}$ to conclude that
\begin{equation}
\label{eq:Mabcdefghi}
\exp\Big(\sum_{n\leq e^{y_0} }\frac{\Re(a_f(n) n^{-i\phi})\Lambda(n)}{n\log n}\Big)\gg  \frac{y_0^{\kappa-1/10}}{(\log y_0)^2}.
\end{equation}
The bound \eqref{switch} now follows from \eqref{eq:S11}, \eqref{eq:sbggjgjgjgjgjgjjggjgj}, \eqref{eq:Mabcdefghi}, and the bound $N\ll y_0^{1/10}$.

Suppose that $a_f(n)$ is real-valued.  The Cauchy--Schwarz inequality, \eqref{eqn:cosine_asymptotic}, \eqref{eqn:generalized_mertens}, \eqref{eq:MNlb}, and~\eqref{eq:Mabcdefgh} yield
\begin{align*}
&\kappa \log{y_0}-\log{N}-2\log\log N\\
&\leq \sum_{n\leq e^{y_0} }\frac{a_f(n)\cos(\phi\log n)\Lambda(n)}{n\log n}\\
&\leq 
\Big(\sum_{n\leq e^{y_0} }\frac{a_f(n)^2\Lambda(n)}{n\log n}\Big)^{\frac{1}{2}}\Big(\sum_{n\leq e^{y_0} }\frac{\cos^2(\phi\log n)\Lambda(n)}{n\log n}\Big)^{\frac{1}{2}}\\
&\leq (\kappa^2\log y_0+O(1))^{\frac{1}{2}}\Big(\frac{2}{\pi}\log y_0+\Big(1-\frac{2}{\pi}\Big)\log\max\{\Cr{Lem2.4_large}(\log y_0)^2,|\phi|^{-1}\}+O(1)\Big)^{\frac{1}{2}}.
\end{align*}
In particular, the bounds
\[
\max\{\Cr{Lem2.4_large}(\log y_0)^2,|\phi|^{-1}\}\gg y_0 (N (\log N)^2)^{-\frac{2\pi}{(\pi-2)\kappa}}\gg y_0 N^{-6/\kappa}
\]
hold.  Since $N\le y_0^{1/10}$, we would obtain a contradiction if $\Cr{Lem2.4_large}(\log y_0)^2$ were the maximum.  This yields \eqref{eq:phifreal}, as desired.
\end{proof}

\section{$L$-function estimates}

Let $L(s,\pi)\in\mathfrak{S}(m)$ have analytic conductor $\CC$.  It is immediate that if $\kappa$ is the constant in \eqref{eqn:weak_ramanujan}, then the multiplicative function $\lambda_{\pi}(n)$ lies in the class $\mathcal{C}_w(\kappa)$.  Thus, all of the results from \cref{sec:mean_values} apply with $f=\lambda_{\pi}$.  Throughout this section, unless otherwise specified, all implied constants will depend at most on $\kappa$, $A_0$, and $m$.

\subsection{Convexity bounds for $L(s,\pi)$}
We require bounds for $L(s,\pi)$ just to the right of the line $\re(s)=1$.  By \cite[(5.5)]{ST}, we have that
\begin{equation}
	\label{eqn:convexity_w}
	|L(\tfrac{1}{2},\pi)|\ll|L(\tfrac{3}{2},\pi)|^2 \CC^{\frac{1}{4}}.
\end{equation}
The bound $L(\frac{3}{2}+it,\pi)\ll 1$ follows from \cref{lem:sound_weak_lemma_3.2} applied to $f(n)=\lambda_{\pi}(n)$.  If we change $\frac{1}{2}$ to $\frac{1}{2}+it$ on the left  side  of \eqref{eqn:convexity_w}, the effect on the right  side  is that the numbers $\mu_{\pi}(j)$ shift by $it$.  Thus, $\CC$ is replaced by $\CC(it)$, and we conclude via the triangle inequality that
\begin{equation}
	\label{eqn:convexity_0}
	|L(\tfrac{1}{2}+it,\pi)|\ll \CC^{\frac{1}{4}}(1+|t|)^{\frac{m}{4}}.
\end{equation}
We use \cref{lem:sound_weak_lemma_3.2} again to achieve the bound $L(1+\frac{1}{\log \CC}+it,\pi)\ll (\log \CC)^{\kappa}$.  Thus, by the Phragm{\'e}n-Lindel{\"o}f principle, we find that
\begin{equation}
	\label{eqn:convexity_1}
	|L(\sigma+it,\pi)|\ll \CC^{\frac{1-\sigma}{2}}(\log \CC)^{\kappa(2\sigma-1)} (1+|t|)^{\frac{m(1-\sigma)}{2}},\qquad 1/2\leq\sigma\leq 1.
\end{equation}

\subsection{Partial sums and $L$-functions}

We now prove a standard estimate on the partial sums of $\lambda_{\pi}(n)$ when $L(s,\pi)\in\mathfrak{S}(m)$.  We will make no attempt to optimize the exponents.

\begin{lemma}
	\label{lem:partial_sums_L-function}
	If $L(s,\pi)\in\mathfrak{S}(m)$ (in particular, $L(s,\pi)$ has no pole), then
	\[
	\sum_{n\leq x}\lambda_{\pi}(n)\ll x^{1-\frac{1}{m+14}}((\log x)^{\frac{\kappa^2}{2}}+\CC^{\frac{1}{4}}),\qquad x\geq 3.
	\]
\end{lemma}
\begin{proof}
	Let $c>1$, $y>0$, and $\lambda>0$.  Let $K\geq 0$ be an integer.  A straightforward Perron integral calculation (see \cite[Equation 4.1]{Sound_weak}) reveals that
	\[
	\sum_{n\leq x}\lambda_{\pi}(n)=\frac{1}{2\pi i}\int_{c-i\infty}^{c+i\infty}L(s,\pi)\frac{x^s}{s}\Big(\frac{e^{\lambda s}-1}{\lambda s}\Big)^K ds+O\Big(\sum_{x<n\leq e^{K\lambda}x}|\lambda_{\pi}(n)|\Big).
	\]
	We choose $K=\lfloor m/4\rfloor+3$ and $\lambda=x^{-2/(m+14)}$.  It suffices to take $x$ large enough so that $\lambda\leq 1/K$, in which case  $e^{K\lambda}-1\leq 2K\lambda$.  Thus, using these bounds and the Cauchy--Schwarz inequality, we deduce from \cref{lem:sound_weak_lemma_3.3} that
	\[
	\sum_{x<n\leq e^{K\lambda}x}|\lambda_{\pi}(n)|\ll (e^{K\lambda}-1)^{\frac{1}{2}}x(\log x)^{\frac{\kappa^2}{2}}\ll \lambda^{\frac{1}{2}} x(\log x)^{\frac{\kappa^2}{2}}.
	\]
	
	It follows from \eqref{eqn:convexity_1} that if we push the contour to the line $\re(s)=1/2$, then
	\[
	\frac{1}{2\pi i}\int_{c-i\infty}^{c+i\infty}L(s,\pi)\frac{x^s}{s}\Big(\frac{e^{\lambda s}-1}{\lambda s}\Big)^K ds \ll \CC^{\frac{1}{4}}x^{\frac{1}{2}}\lambda^{-K} \int_{-\infty}^{\infty}(1+|t|)^{m/4-K}dt  \ll \CC^{\frac{1}{4}}x^{\frac{1}{2}}\lambda^{-K}.
	\]
	We obtain the lemma by combining our two estimates.
\end{proof}

The next lemma relates the mean values of $\lambda_{\pi}(n)/n^{i\phi}$ to a suitable average of $L(s,\pi)$.

\begin{lemma}
\label{lem:plancherel}
Let $L(s,\pi)\in\mathfrak{S}(m)$. Let $\phi\in\mathbb{R}$ and $T\geq 0$.  If $0\leq \lambda<\frac{1}{m+15}$, then
\[
\sqrt{2\pi T}\int_{-\infty}^{\infty}\Big(\frac{1}{e^y}\sum_{n\leq e^y}\frac{\lambda_{\pi}(n)}{n^{i\phi}}\Big)\exp\Big(\lambda y-\frac{T}{2}y^2\Big)dy=\int_{-\infty}^{\infty}\frac{L(1-\lambda+i\phi+it,\pi)}{1-\lambda+it}\exp\Big(-\frac{t^2}{2T}\Big)dt.
\] 
\end{lemma}
\begin{proof} 
By \cref{lem:partial_sums_L-function} and partial summation, we have that if $\re(s)\geq 1-\frac{1}{m+15}$, then
\[
\lim_{x\to\infty}\sum_{n\geq x} \frac{\lambda_{\pi}(n)}{n^s}=s\lim_{x\to\infty}\int_x^{\infty}\frac{-1}{t^{1+s}}\Big(\sum_{n\leq t}\lambda_{\pi}(n)\Big)dt\ll_{\pi} |s|\lim_{x\to\infty}\int_x^{\infty}\frac{t^{1-\frac{1}{m+14}}(\log t)^{\kappa^2/2}}{t^{1+\re(s)}}dt=0.
\]  
Therefore, for $0\leq \lambda<\frac{1}{m+15}$, we can compute the Fourier transform
\begin{align*}
\int_{-\infty}^{\infty}\Big(\frac{1}{e^y}\sum_{n\leq e^y}\frac{\lambda_{\pi}(n)}{n^{i\phi}}\Big)\exp(\lambda y)e^{-it}dt=\sum_{n\geq 1}\frac{\lambda_{\pi}(n)}{n^{i\phi}}\int_{\log n}^{\infty}e^{y(\lambda-1-it)}dy=\frac{L(1-\lambda+i\phi+it,\pi)}{1-\lambda+it}.
\end{align*}
The Fourier transform of $\exp(-\frac{T}{2}y^2)$ is
\[
\int_{-\infty}^{\infty}\exp\Big(-\frac{T}{2}y^2-iy\psi\Big)dy=\frac{\sqrt{2\pi}}{\sqrt{2T}}\exp\Big(\frac{-\psi^2}{2T}\Big),
\] 
so the lemma follows from Plancherel's formula.
\end{proof}

We are now ready to establish a variant of \cite[Proposition 3.1]{GS_JEMS}, which relates large partial sums of the Dirichlet coefficients of $L$-functions $L(s,\pi)\in\mathfrak{S}(m)$ to large values of of $L(s,\pi)$.  Let $\ell_{\kappa}$ be as in \eqref{eqn:ell_kappa_def}.
\begin{proposition}
\label{prop:zerosfromabove}
Let $L(s,\pi)\in\mathfrak{S}(m)$ with analytic conductor $\CC$, and let $\delta\in(0,\frac{1}{2}]$.  There exist constants $\Cl[abcon]{GS_Prop3.1},\Cl[abcon]{GS_Prop3.1_2}>0$ (depending on $\kappa$, $A_0$, $m$, and $\delta$) such that the following is true.  Let $\phi$ be as in \cref{lem:twist}. Let
\begin{equation}
\label{eqn:ranges_3.3}
\CC\geq \Cr{GS_Prop3.1},\qquad y_0\in[(\log \CC)^{1-1/(50m)},\tfrac{1}{2}\log\CC],\qquad N\in[1,y_0^{1/(100m)}].
\end{equation}
If
\begin{equation}
\label{eqn:Ndef_3.3}
\Big|\sum_{n\leq e^{y_0}}\lambda_{\pi}(n)\Big|=\frac{e^{y_0}y_0^{\kappa-1}}{N}.
\end{equation}
and
\begin{align}
\label{eq:A-y}
\Cr{GS_Prop3.1_2}\frac{N^{2/\ell_{\kappa}}}{y_0}\Big(\frac{\log\CC}{y_0}\Big)^{1-2\delta}\leq \lambda< \frac{1}{m+15},
\end{align}
then there exists
\begin{equation}
\label{eqn:xi_range}
\xi\in[-\lambda (y_0^{-1}\kappa\log{\CC})^{\delta},\lambda (y_0^{-1}\kappa\log{\CC})^{\delta}]
\end{equation}
such that
\begin{equation}
\label{eqn:3.10}
\frac{\lambda y_0}{4}\leq \log\Big|\frac{L(1-\lambda+i(\phi+\xi),\pi)}{L(1+\lambda+i(\phi+\xi),\pi)}\Big|.
\end{equation}
\end{proposition}
\begin{proof}
Let $\phi$ be as in \cref{lem:twist}, and let $T>0$. For any $y\geq 0$, we have uniformly 
\begin{equation}\label{variation}
\Big|\frac{1}{e^y}\sum_{n\leq e^{y}}\frac{\lambda_{\pi}(n)}{n^{i\phi}}-\frac{1}{e^{y_0}}\sum_{n\leq e^{y_0}}\frac{\lambda_{\pi}(n)}{n^{i\phi}}\Big|\ll  \frac{1+|y-y_0|^{\ell_{\kappa}}}{y_0^{\ell_{\kappa}}}y^{\kappa-1}.
\end{equation}
(Indeed, if $|y-y_0|\geq y_0/3$ then this inequality immediately follows from \cref{lem:triv_bound_holds}.	 For $|y-y_0|\leq y_0/3,$ we invoke \cref{lem:lipshitz} after appropriate substitutions.) By \eqref{variation}, we see that
\begin{align*}
&\sqrt{2\pi T}\int_{-\infty}^{\infty}\Big(\frac{1}{e^y}\sum_{n\leq y}\frac{\lambda_{\pi}(n)}{n^{i\phi}}\Big)\exp\Big(\lambda y-\frac{T}{2}y^2\Big)dy\\
&=\sqrt{2\pi T}\exp\Big(\frac{\lambda y_0}{2}\Big)\int_{-\infty}^{\infty}
\Big(\frac{1}{e^{y_0}}\sum_{n\leq e^{y_0}}\frac{\lambda_{\pi}(n)}{n^{i\phi}}+O\Big(\frac{1+|y-y_0|^{\ell_{\kappa}}}{y_0^{\ell_{\kappa}}}|y|^{\kappa-1}\Big)\Big)e^{-\frac{T}{2}(y-y_0)^2}dy\\
&=2\pi\exp\Big(\frac{\lambda y_0}{2}\Big)\Big[\frac{1}{e^{y_0}}\sum_{n\leq e^{y_0}}\frac{\lambda_{\pi}(n)}{n^{i\phi}} + O\Big(\sqrt{T}\int_{-\infty}^{\infty}\frac{1+|y-y_0|^{\ell_k}}{y_0^{\ell_k}}|y|^{\kappa-1}e^{-\frac{T}{2}(y-y_0)^2}dy\Big)\Big].
\end{align*}
We apply \cref{lem:twist} to the main term on the last line of the above display; consequently, there exists a constant $\Cl[abcon]{prop3.3_1}=\Cr{prop3.3_1}(\kappa,m,\delta,A_0)>0$ such that
\begin{equation}
\label{eqn:prop3.3}
\begin{aligned}
&\Big|\frac{\sqrt{2\pi T}}{2\pi\exp(\frac{\lambda y_0}{2})}\int_{-\infty}^{\infty}\Big(\frac{1}{e^{y}}\sum_{n\leq e^{y}}\frac{\lambda_{\pi}(n)}{n^{i\phi}}\Big)\exp\Big(\lambda y-\frac{T}{2}y^2\Big)dy-\frac{(1+i\phi)}{e^{y_0(1+i\phi)}} \sum_{n\leq e^{y_0}}\lambda_{\pi}(n)\Big|\\
&\leq \Cr{prop3.3_1}\Big[y_0^{\kappa-1-\frac{1}{5}}+\sqrt{T}\int_{-\infty}^{\infty}\frac{1+|y-y_0|^{\ell_k}}{y_0^{\ell_k}}\frac{|y|^{\kappa-1}}{e^{\frac{T}{2}(y-y_0)^2}}dy\Big].
\end{aligned}\hspace{-1mm}
\end{equation}
Using the substitution $t=y-y_0$ and the bounds
\[
N\geq 1,\qquad |y|^{\kappa-1}\leq (|y-y_0|+y_0)^{\kappa-1}\leq 2^{\kappa-2}(|y-y_0|^{\kappa-1}+y_0^{\kappa-1}),
\]
we find that there exists a constant $\Cl[abcon]{Tchoice}=\Cr{Tchoice}(\kappa,m,A_0,\delta)>0$ such that if
\[
T = \Cr{Tchoice}y_0^{-2}N^{2/\ell_{\kappa}},
\]
then \eqref{eqn:prop3.3} is at most $\pi y_0^{\kappa-1}/N$.  Therefore, by \cref{lem:plancherel} and \eqref{eqn:Ndef_3.3}, we conclude that
\begin{align*}
\pi\exp\Big(\frac{\lambda y_0}{2}\Big)\frac{y_0^{\kappa-1}}{N}&\leq\Big|\sqrt{2\pi T}\int_{-\infty}^{\infty}\Big(\frac{1}{e^y}\sum_{n\leq y}\frac{\lambda_{\pi}(n)}{n^{i\phi}}\Big)\exp\Big(\lambda y-\frac{T}{2}y^2\Big)dy\Big|\\
&=\Big|\int_{-\infty}^{\infty}\frac{L(1-\lambda+i\phi+it,\pi)}{1-\lambda+it}\exp\Big(-\frac{t^2}{2T}\Big)dt\Big|\\
&\leq \max_{\xi\in\mathbb{R}} \frac{|L(1-\lambda+i(\phi+\xi),\pi)|}{|1-\lambda+i\xi|}\exp\Big(-\frac{\xi^2}{4T}\Big)\int_{\mathbb{R}}\exp\Big(\frac{-t^2}{4T}\Big)dt\\
&=2\sqrt{\pi T}\max_{\xi\in\mathbb{R}} \frac{|L(1-\lambda+i(\phi+\xi),\pi)|}{|1-\lambda+i\xi|}\exp\Big(-\frac{\xi^2}{4T}\Big).
\end{align*}
With our choice of $T$, we may simplify the above display to
\begin{equation}
\label{eqn:max_xi}
\sqrt{\frac{\pi}{4\Cr{Tchoice}}}\frac{y^{\kappa}_0}{N^{1+1/\ell_{\kappa}}}\exp\Big(\frac{\lambda y_0}{2}\Big)\leq  \max_{\xi\in\mathbb{R}} \frac{|L(1-\lambda+i(\phi+\xi),\pi)|}{|1-\lambda+i\xi|}\exp\Big(-\frac{y^2_0\xi^2 }{4\Cr{Tchoice}N^{2/\ell_{\kappa}}}\Big).
\end{equation}

The maximum in~\eqref{eqn:max_xi} occurs at a value of $\xi\in\R$ satisfying \eqref{eqn:xi_range}.  Indeed, suppose otherwise, to the contrary.  We use~\eqref{eqn:convexity_1} and \eqref{eqn:ranges_3.3} to obtain the bound
\begin{align*}
\max_{\xi\in\mathbb{R}} \frac{|L(1-\lambda+i(\phi+\xi),\pi)|}{|1-\lambda+i\xi|}\exp\Big(\frac{-y^2_0\xi^2 }{4\Cr{Tchoice}N^{2/\ell_{\kappa}}}\Big)&\ll \CC^{\frac{\lambda}{2}}(\log \CC)^{\kappa}\exp\Big(-\frac{(\lambda y_0)^2(y_0^{-1}\kappa\log \CC)^{2\delta}}{4\Cr{Tchoice}N^{2/\ell_{\kappa}}}\Big).
\end{align*}
This upper bound, together with the lower bound in \eqref{eqn:max_xi}, contradicts \eqref{eq:A-y} when $\Cr{GS_Prop3.1}$ and $\Cr{GS_Prop3.1_2}$ are made sufficiently large.  We may therefore choose $\xi$ satisfying \eqref{eqn:xi_range} so that the maximum in~\eqref{eqn:max_xi} is attained.  For such a choice of $\xi$, \eqref{eqn:max_xi} yields
\begin{equation}
\label{eqn:nonvanish}
\sqrt{\frac{y_0}{\lambda}}\frac{y_0^{\kappa-1}}{N}\exp\Big(\frac{\lambda y_0}{2}\Big)\ll |L(1-\lambda+i(\phi+\xi),\pi)|.
\end{equation}
In light of the bound $|L(1+\lambda+i(\phi+\xi),\pi)|\ll \lambda^{-\kappa}$ per \cref{lem:sound_weak_lemma_3.2}, we conclude that
\begin{align*}
\frac{(\lambda y_0)^{\kappa-\frac{1}{2}}}{N}\exp\Big(\frac{\lambda y_0}{2}\Big)\ll \frac{|L(1-\lambda+i(\phi+\xi),\pi)|}{|L(1+\lambda+i(\phi+\xi),\pi)|}.
\end{align*}
Taking logarithms, applying \eqref{eq:A-y}, and taking $\Cr{GS_Prop3.1}$ and $\Cr{GS_Prop3.1_2}$ to be sufficiently large, we deduce the proposition.
\end{proof}

\section{A ``hybrid Euler--Hadamard product''}
\label{sec:explicit}

Following \cite[Lemma 3.1]{GS_JEMS} (see also \cite{FGGHT}), one can use the Hadamard factorization of $L(s,\pi)$ to relate the right hand side of \eqref{eqn:3.10} to the sum
\[
\sum_{\rho}\frac{\lambda^2}{|1+\lambda+it-\rho|^2}
\]
over nontrivial zeros $\rho$ of $L(s,\pi)$.  This sum can be bounded using the triangle inequality and basic properties of the zeros of $L(s,\pi)$.  In order to quantitatively improve upon the results in \cite{FGGHT,GS_JEMS}, we deviate from this approach and prove the following result.
\begin{proposition}
\label{prop:HadamardEuler}
Let $u\colon \R\to[0,\infty)$ be an infinitely differentiable function with compact support contained in the interval $[1,e]$ such that $\int_0^{\infty}u(t)dt=1$.  For $s=\sigma+it$, let $\hat{u}(s)=\int_0^{\infty}u(t)t^{s-1}dt$ be the Mellin transform of $u$.  For $x\geq 0$, define $v(x)=\int_x^{\infty}u(t)dt$.

Let $L(s,\pi)\in\mathfrak{S}(m)$.  If $k\geq 3$, $K>0$, $X\geq e^{3m^4}$, 
\[
|\Re{(s-1)}|\ll \frac{1}{\log X},
\]
and $s$ is not a zero of $L(s,\pi)$, then
\begin{align*}
\log{L(s,\pi)}=\sum_{n=2}^{\infty}\frac{a_{\pi}(n)\Lambda(n)}{n^{s}\log n}v(e^{\frac{\log{n}}{\log X}})-\sum_{|s-\rho|\leq K/\log X}U((s-\rho)\log X) +O\Big(\frac{e^{4k}\log \CC(it)}{K^{k-2}\log X}\Big).
\end{align*} 
\end{proposition}

One may think of \cref{prop:HadamardEuler} as a variant of the hybrid ``Euler--Hadamard product'' for the Riemann zeta function proved by Gonek, Hughes, and Keating \cite[Theorem 1]{GonekHughesKeating}.  See also Bombieri and Hejhal \cite[Lemma 1]{BombieriHejhal}.  In contrast with these results, \cref{prop:HadamardEuler} maintains uniformity in $\CC$, which is crucial for our results.  Also, \cref{prop:HadamardEuler} focuses on $s$ near the line $\Re(s)=1$, whereas the results in \cite{BombieriHejhal,GonekHughesKeating} focus on $s$ near the line $\re(s)=\frac{1}{2}$.

\subsection{Notation and hypotheses}
\label{subsec:notation_hypotheses}

Let $s=\sigma+it$, and let $X\geq e^{3m^4}$.  For a function $g\colon[0,\infty)\to\R$, we define the Mellin transform
\[
\hat{g}(s)=\int_0^{\infty}g(t)t^{s-1}dt.
\]
Let $u\colon\R\to[0,\infty)$ be an infinitely differentiable function whose support is a compact subset of $[1,e]$.  Given $x\geq 0$ define
\begin{align}
\label{eq:vdef}
v(x) = \int_x^{\infty}u(t)dt.
\end{align}
For $z\in \mathbb{C}$ define 
\begin{align}
\label{eq:Udef}
E_1(z)=\int_{z}^{\infty}\frac{e^{-w}}{w}dw,\qquad U(z)=\int_{0}^{\infty}u(x)E_1(z\log x)dx.
\end{align}
By \cite[Equation 12]{GonekHughesKeating}, if $s_0=\sigma_0+it_0$ and $z=x+iy$ are complex numbers, then
\begin{align}
\label{eq:U1}
U((s_0-z)\log X)=\int_{\sigma_0}^{\infty}\frac{\hat{u}(1-(\sigma+it_0-z)\log X)}{\sigma+it_0-z}d\sigma.
\end{align}

\subsection{Preliminary estimates}

We begin with some relationships between $u$ and $v$.

\begin{lemma}
\label{lem:u1}
Let $k\geq 1$ be an integer.  If $u$ and $v$ are as in \cref{subsec:notation_hypotheses} and $s=\sigma+it$, then
\begin{align*}
\hat{v}(s)&=\hat{u}(s+1)/s,\quad |\hat{u}^{(k)}(s)|\leq \hat{u}(\sigma),\quad |\hat{u}(s)|\leq\max_{x\in\R}|{u}^{(k)}(x)|e^{\max\{\sigma,0\}+4k}(1+|s|)^{-k}.
\end{align*}
\end{lemma}
\begin{proof}
See the proof of \cite[Lemma 1]{BombieriHejhal}.
\end{proof}
Our next result provides a useful estimate for $U$ in \eqref{eq:U1}.
\begin{lemma}
\label{lem:Ularge}
Let  $X\geq e^{3m^4}$.  Let $s_0=\sigma_0+it_0$ and $z=x+iy$ be complex numbers satisfying $\im(s_0)\neq \im(z)$.  Assume that there exists a constant $\Cl[abcon]{Lem4.3}>0$ such that
\begin{align}
\label{eq:s0zassumption}
\re(s_0-z)> -\frac{\Cr{Lem4.3}}{\log X}.
\end{align}
If $k\geq 1$, then
\begin{align}
\label{eq:U1-11}
U((s_0-z)\log X)\ll \frac{e^{4k}}{(\Re(s_0-z)+|\im(s_0-z)|)^k (\log X)^{k}},
\end{align}
\end{lemma}
\begin{proof}
Our argument is similar to the analysis in \cite[pp. 842]{BombieriHejhal}.
The assumption \eqref{eq:s0zassumption} combined with Lemma \ref{lem:u1} implies for any integer $k\geq 1$ and real number $\sigma$ satisfying  $\sigma-\re(z)\geq \re(s_0)$ that
\begin{align*}
\hat{u}(1-(\sigma+i\im(s_0-z)\log X)\ll \frac{e^{4k}}{(|\sigma-x|+|\im(s_0-z)|)^{k}(\log X)^{k}}.
\end{align*}
By \eqref{eq:U1}, we conclude that
\begin{align*}
 U((s_0-z)\log X)&=\int_{\sigma_0+it_0}^{\infty+it_0}\frac{\hat{u}(1-(s-z)\log X)}{s-z}ds \\
&\ll \frac{e^{4k}}{(\log X)^{k}}\int_{\sigma_0}^{\infty}\frac{1}{(|\sigma-x|+|\im(s_0-z)|)^{k+1}}d\sigma\\
&\ll \frac{e^{4k}}{((\Re(s_0-z)+|\im(s_0-z)|)^k(\log X)^{k}}.\qedhere
\end{align*}
\end{proof}

Our next result relates $-(L'/L)(s,\pi)$ to a smoothed sum of its Dirichlet coefficients plus a rapidly decaying sum over the nontrivial zeros of $L(s,\pi)$ via contour integration.

\begin{lemma}
\label{lem:explicit1}
Let $L(s,\pi)\in\mathfrak{S}(m)$ and $X\geq e^{3m^4}$.  Let $\rho=\beta+i\gamma$ denote a nontrivial zero of $L(s,\pi)$.  If $s=\sigma+it$ is not a zero of $L(s,\pi)$, $k\geq 3$ is an integer, and $\sigma\geq 1-\frac{1}{2m}$, then
\begin{align*}
\log L(s,\pi)=\sum_{n=2}^{\infty}\frac{a_{\pi}(n)\Lambda(n)}{n^{s}\log{n}}v(e^{\frac{\log{n}}{\log X}})-\sum_{\rho}U((s-\rho)\log X) +O\Big(\frac{e^{4k}}{(\log X)^{k}}\Big).
\end{align*}
\end{lemma}
\begin{proof}
By Lemma \ref{lem:u1}, if $c>0$ is sufficiently large with respect to $s$ and $X$, then
\begin{align*}
\frac{1}{2\pi i}\int_{c-i\infty}^{c+i\infty}\frac{\hat{u}(w+1)}{w n^{w/\log X}}dw=v(e^{\frac{\log{n}}{\log X}}).
\end{align*}
By \eqref{eqn:logarithmic_derivative}, it follows that
\begin{equation}
\label{eqn:a(n)Mellin}
\sum_{n=1}^{\infty}\frac{a_{\pi}(n)\Lambda(n)}{n^{s}}v(e^{\frac{\log{n}}{\log X}})=\frac{1}{2\pi i}\int_{c-i\infty}^{c+i\infty}-\frac{L'}{L}\Big(s+\frac{w}{\log X},\pi\Big)\frac{\hat{u}(w+1)}{w}dw.
\end{equation}
Using the rapid decay of $\hat{u}$, we push the contour to the left and find that \eqref{eqn:a(n)Mellin} equals
\begin{align*}
-\frac{L'}{L}(s,\pi)+\sum_{\rho}\frac{\hat{u}(1-(s-\rho)\log X)}{s-\rho} +\sum_{j=1}^{m}\sum_{n=0}^{\infty}\frac{\hat{u}(1-(s+2n+\mu_{\pi}(j))\log X)}{s+2n+\mu_{\pi}(j)},
\end{align*}
where $\rho$ runs through the nontrivial zeros of $L(s,\pi)$ and the numbers $\mu_{\pi}(j)$ are as in \eqref{eqn:gamma_factor}.

We will now prove that if $\sigma\geq 1-\frac{1}{2m}$, then
\begin{equation}
\label{eqn:trivial_zero_bound}
\sum_{j=1}^{m}\sum_{n=0}^{\infty}\frac{\hat{u}(1-(2n+\mu_{\pi}(j)+s)\log X)}{2n+\mu_{\pi}(j)+s}\ll \frac{1}{(\log X)^{k}|s|^k}.
\end{equation}
It follows from \eqref{eqn:weak_selberg}, our range of $X$, and the constraint $\sigma\geq 1-\frac{1}{2m}$ that
\[
\Re(1-(2n+\mu_{\pi}(j)+s)\log X)\leq 1-\Big(2n-1+\frac{1}{m}+\sigma\Big)\log X\leq -\frac{(\sigma+2n)\log X}{4m}.
\]
The bound \eqref{eqn:trivial_zero_bound} now follows from \eqref{lem:u1}, \eqref{eq:U1}, and \cref{lem:Ularge}.  We have arrived at
\begin{equation}
\label{eqn:logderiv}
-\frac{L'}{L}(s,\pi)=\sum_{n=1}^{\infty}\frac{a_{\pi}(n)\Lambda(n)}{n^{s}}v(e^{\frac{\log{n}}{\log X}})-\sum_{\rho}\frac{\hat{u}(1-(s-\rho)\log X)}{s-\rho}+O\Big(\frac{e^{4k}}{(\log X)^{k}|s|^k}\Big).
\end{equation}
To finish, integrate \eqref{eqn:logderiv} along $\{z\in\mathbb{C}\colon \Re(z)\geq\sigma,~\im(z)=t\}$.
\end{proof}

\subsection{Proof of Proposition \ref{prop:HadamardEuler}}

To establish sufficient decay for the sum over $\rho$ in Lemma~\ref{lem:explicit1}, we require estimates for the local distribution of zeros of $L(s,\pi)$.

\begin{lemma}
\label{lem:rvm}
Let $L(s,\pi)\in\mathfrak{S}(m)$.  If $t\in\R$, then
\begin{align*}
\#\{ \rho=\beta+i\gamma  :  L(\rho,\pi)=0, \ 0\leq \beta \leq 1, \  t\leq \gamma \leq t+1  \} \ll \log \CC(it).
\end{align*}
\end{lemma}
\begin{proof}
This follows from \cite[Proposition 5.7]{IK}.
\end{proof}

\begin{lemma}
\label{lem:zeroc}
Let $L(s,\pi)\in\mathfrak{S}(m)$.  If $t$ and $r$ are real numbers such that
\begin{align*}
\frac{1}{\log \CC(it)}\leq r \leq \frac{3}{4},
\end{align*}
then $\#\{\rho\colon L(\rho,\pi)=0,~|\rho-(1+it)|\leq r \}\ll r\log \CC(it)$.
\end{lemma}
\begin{proof}
This follows from \cite[Lemma 3.1]{ST}.
\end{proof}

\begin{lemma}
\label{lem:nontrivial}
If $k\geq 1$ is an integer, $K> 0$, $X\geq e^{3m^4}$, and $s=\sigma+it$, then
\begin{align*}
\sum_{\substack{\rho \\ |s-\rho|\geq K/\log X}}U((s-\rho)\log X)\ll \frac{e^{4k}\log \CC(it)}{K^{k-2}\log X},\qquad |\Re{(s-1)}|\ll \frac{1}{\log X}.
\end{align*}
\end{lemma}
\begin{proof}
If $\rho$ is a nontrivial zero such that $|s-\rho|\geq K/\log X$, then
\begin{align*}
 |\sigma-\beta| \geq \frac{K}{\log X} \qquad \text{or} \qquad |t-\gamma|\geq \frac{K}{\log X}.
\end{align*}
By Lemma \ref{lem:Ularge}, it follows that
\begin{equation}
\label{eqn:Ubound_near_1}
U((s-\rho)\log X)\ll \frac{e^{4k}}{(K+|t-\gamma|(\log X))^{k}}.
\end{equation}
Together, \cref{lem:rvm} and \eqref{eqn:Ubound_near_1} imply that
\begin{align*}
\sum_{\substack{\rho \\ |s-\rho|\geq K/\log X \\ |t-\gamma|\geq 1}}U((s-\rho)\log X)&\ll  e^{4k}\sum_{\substack{j\in \Z \\ j\neq 0}}\frac{\#\{ \rho\colon |s-\rho|\geq K/\log X,~t+j\leq \gamma < t+j+1\}}{(K+|j|(\log X))^{k}}\\
&\ll e^{4k}\sum_{\substack{j\in \Z \\ j\neq 0}}\frac{\log\CC(it)+\log(|j|+2)}{(K+|j|(\log X))^{k}}\\
&\ll e^{4k}\frac{\log\CC(it)+K+\log X}{(K+\log X)^{k-1}},
\end{align*}
which leads us to
\begin{equation}
\label{eq:Ubounds1}
\begin{aligned}
\sum_{\substack{\rho \\ |s-\rho|\geq K/\log X}}U((s-\rho)\log X)= \sum_{\substack{\rho \\ |s-\rho|\geq K/\log X \\ |t-\gamma|\leq 1}}U((s-\rho)\log X) +O\Big(\frac{e^{4k}\log\CC(it)}{(K+\log X)^{k-2}}\Big).\hspace{-3mm}
\end{aligned}
\end{equation}
We partition the sum on the right hand side of \eqref{eq:Ubounds1} as follows:
\begin{align}
\label{eq:S1s1s1}
\sum_{\substack{\rho \\ |s-\rho|\geq K/\log X \\ |t-\gamma|\leq 1}}U((s-\rho)\log X)=\sum_{j=1}^{\infty}S_j,\qquad S_j=\sum_{\substack{\rho \\ \frac{2^{j-1}K}{\log X}\leq |s-\rho|< \frac{2^{j}K}{\log X} \\ |t-\gamma|\leq 1}}U((s-\rho)\log X).\hspace{-1mm}
\end{align}
By \cref{lem:Ularge,lem:zeroc}, we have that
\begin{align*}
S_j\ll e^{4k}\frac{\#\{\rho\colon |s-\rho|\leq 2^{j}K/\log X \}}{(2^{j}K)^{k+1}}\ll \frac{\log \CC(it)}{\log X}\frac{e^{4k}}{(2^{j}K)^{k}}.
\end{align*}
We insert this estimate into \eqref{eq:S1s1s1}, thus obtaining
\begin{align}
\label{eq:Ubounds2}
\sum_{\substack{\rho \\ |s-\rho|\geq K/\log X \\ |t-\gamma|\leq 1}}U((s-\rho)\log X)\ll e^{4k}\frac{\log \CC(it)}{(\log X)K^{k}}.
\end{align}
The lemma follows from \eqref{eq:Ubounds1} and \eqref{eq:Ubounds2}.
\end{proof}

\begin{proof}[Proof of \cref{prop:HadamardEuler}]
This follows from  \cref{lem:explicit1,lem:nontrivial}.  We have assumed that $k\geq 3$ is an integer, but by replacing $k$ with $\lfloor k\rfloor+1$, we see that \cref{prop:HadamardEuler} is true for any $k\geq 3$.
\end{proof}

\section{Proof of Theorem \ref{thm:main}}
\label{sec:main_proof}

Let $L(s,\pi)\in\mathfrak{S}(m)$, and recall the notation and hypotheses of \cref{prop:zerosfromabove,prop:HadamardEuler}.  By \cref{prop:zerosfromabove}, if \eqref{eq:A-y} holds, then
\begin{align}
\label{eq:sss1-9}
\frac{\lambda y_0}{4}\leq \log\Big|\frac{L(1-\lambda+i(\phi+\xi),\pi)}{L(1+\lambda+i(\phi+\xi),\pi)}\Big|,\qquad \textup{where}\quad |\xi|\leq \lambda( y_0^{-1}\kappa\log{\CC})^{\delta}.
\end{align}
Let $\Cl[abcon]{small}=\Cr{small}(\kappa,A_0,m,\delta)\in(0,1)$ be a sufficiently small constant, and let
\begin{align}
\label{eq:Xconds}
X = e^{\Cr{small}/\lambda},\qquad y\in[-\lambda,\lambda].
\end{align}
We will apply \cref{prop:HadamardEuler} to $\log L(1\pm \lambda+i(\phi+\xi),\pi)$; we can apply \cref{prop:HadamardEuler} to $L(1-\lambda+i(\phi+\xi),\pi)$ because of \eqref{eqn:nonvanish}.  To begin, let $k$ and $B$ be large integers depending at most on $\kappa$, $A_0$, $\delta$ and $m$; they are defined in \eqref{eq:11} below. We apply \cref{prop:HadamardEuler} with 
\begin{align}
\label{eqn:Kmain1+lamdalb}
K=\Big(\frac{B\kappa \log{\CC}}{\lambda y_0\log X}\Big)^{1/(k-2)}.
\end{align}
It follows that if
\begin{align}
\label{eq:CC-1-3993}
s=\sigma+it\in \mathbb{C},\qquad |\Re(s-1)|\ll \frac{1}{\log{X}},
\end{align}
then
\begin{align}
\label{eq:Lsp1}
\log{L(s,\pi)}&=\sum_{n=2}^{\infty}\frac{a_{\pi}(n)\Lambda(n)}{n^{s}\log n}v(e^{\frac{\log{n}}{\log X}})-\sum_{|s-\rho|\leq K/\log X}U((s-\rho)\log X) +O\Big(\frac{\lambda y_0 e^{4k}}{B}\Big).\hspace{-2.9mm}
\end{align}

We will apply~\eqref{eq:Lsp1} with $s=1+\lambda+i(\phi+\xi)$ and $s=1-\lambda+i(\phi+\xi)$ then subtract the resulting estimates.  Note that the choice of parameters~\eqref{eq:Xconds} implies that the condition~\eqref{eq:CC-1-3993} is satisfied. Therefore, if $F_1$ and $F_2$ are defined by
\begin{align*}
F_1(y)&:=\sum_{n=1}^{\infty}\frac{a_{\pi}(n)\Lambda(n)}{n^{1+y+i(\phi+\xi)}\log n}v(e^{\frac{\log{n}}{\log X}}),\\
F_2(y)&:=-\sum_{|1+i(\phi+\xi)-\rho|\leq K/\log X}U((1+y+i(\phi+\xi)-\rho)\log X),
\end{align*}
then
\begin{align}
\label{eq:L111}
\log\Big|\frac{L(1-\lambda+i(\phi+\xi),\pi)}{L(1+\lambda+i(\phi+\xi),\pi)}\Big| =  \mathrm{Re}(F_1(-\lambda)-F_1(\lambda)+F_2(-\lambda)-F_2(\lambda)) +O\Big(\frac{\lambda y_0 e^{4k}}{B}\Big).\hspace{-3.8mm}
\end{align}
By the mean value theorem, there exists $y_1\in[-\lambda,\lambda]$ such that 
\[
F_1(-\lambda)-F_1(\lambda)=-2\lambda F^{'}_1(y_1).
\]
By our choices of $u$ and $v$ (see \eqref{eq:vdef}), we have that
\begin{align*}
F_1'(y_1)=-\sum_{n=1}^{\infty}\frac{a_{\pi}(n)\Lambda(n)}{n^{1+y_1+i(\xi+\phi)}}v(e^{\frac{\log n}{\log X}})\ll \sum_{n=1}^{\infty}\frac{|a_{\pi}(n)|\Lambda(n)}{n^{1-\lambda}}v(e^{\frac{\log{n}}{\log X}})\ll \sum_{n\leq X}\frac{|a_{\pi}(n)|\Lambda(n)}{n^{1-\lambda}}.
\end{align*}
Our choice of $X$ in \eqref{eq:Xconds} and the Cauchy--Schwarz inequality ensure that
\begin{align*}
F_1'(y_1)\ll \sum_{n\leq X}\frac{|a_{\pi}(n)|\Lambda(n)}{n}\ll \Big(\sum_{n\leq X}\frac{|a_{\pi}(n)|^2\Lambda(n)}{n}\Big)^{\frac{1}{2}}\Big(\sum_{n\leq X}\frac{\Lambda(n)}{n}\Big)^{\frac{1}{2}}.
\end{align*}
It then follows from \eqref{eqn:weak_ramanujan}, dyadic decomposition, and Mertens's theorem that
\[
F_1(\lambda)-F_1(-\lambda)=2\lambda F_1'(y_1)\ll \lambda \log X \ll 1.
\]
 Thus,~\eqref{eq:L111} reduces to 
\begin{align*}
\log\Big|\frac{L(1-\lambda+i(\phi+\xi),\pi)}{L(1+\lambda+i(\phi+\xi),\pi)}\Big| &=  \mathrm{Re}(F_2(-\lambda)-F_2(\lambda))+O\Big(\frac{\lambda y_0 e^{4k}}{B}+1\Big).\hspace{-2mm}
\end{align*}
It follows from~\eqref{eq:sss1-9} that there exists a constant $\Cl[abcon]{B_scale}=\Cr{B_scale}(A_0,\kappa,m,\delta)\geq 1$ such that
\[
\frac{\lambda y_0}{8}\le   \mathrm{Re}(F_2(-y)-F_2(y))+\Cr{B_scale}\frac{\lambda y_0 e^{4k}}{B}.\hspace{-2mm}
\]
Consequently, there exists a constant $\Cl[abcon]{B_const}=\Cr{B_const}(A_0,\kappa,m,\delta)\geq 1$ such that if
\begin{align}
\label{eq:11}
k=\delta^{-1}+2, \quad B=\Cr{B_const} e^{4k},
\end{align}
then
 \begin{align}
\label{eq:L113}
\frac{\lambda y_0}{16}&\le   \mathrm{Re}(F_2(-\lambda)-F_2(\lambda))\hspace{-2mm}
\end{align}

Next, we estimate the right-hand side of \eqref{eq:L113}.  If $|y|\leq \lambda$, then by~\eqref{eq:U1} and the fundamental theorem of calculus, we find that
\begin{align*}
\frac{\partial}{\partial y}F_2(y)&=-\sum_{|1+i(\phi+\xi)-\rho|\leq K/\log X}\frac{\partial}{\partial y}U((1+y+i(\phi+\xi)-\rho)\log X)\\
&=\sum_{|1+i(\phi+\xi)-\rho|\leq K/\log X}\frac{\hat{u}(1-(1+y+i(\phi+\xi)-\rho)\log X)}{1+y+i(\phi+\xi)-\rho}.
\end{align*}
It follows that
\begin{align}
\label{eqn:F2_int}
F_2(-\lambda)-F_2(\lambda)=\sum_{|1+i(\phi+\xi)-\rho|\leq K/\log X}-\int_{-\lambda}^{\lambda}\frac{\hat{u}(1-(1+y+i(\phi+\xi)-\rho)\log X)}{1+y+i(\phi+\xi)-\rho}dy.\hspace{-1.4mm}
\end{align}
Combining~\eqref{eq:L113} with \eqref{eqn:F2_int}, we arrive at 
\begin{align}
\label{eqn:zerosum_F2}
\frac{\lambda y_0}{16}\le \int_{-\lambda}^{\lambda}\sum_{|1+i(\phi+\xi)-\rho|\leq K/\log X}-\Re{\Big(\frac{\hat{u}(1-(1+y+i(\phi+\xi)-\rho)\log X)}{1+y+i(\phi+\xi)-\rho}\Big)}dy.
\end{align}
Defining
\begin{align*}
S_1(y):=\sum_{\substack{|1+i(\phi+\xi)-\rho|\leq K/\log X \\ |1+y+i(\phi+\xi)-\rho|< 1/\log X}}-\Re{\Big(\frac{\hat{u}(1-(1+y+i(\phi+\xi)-\rho)\log X)}{1+y+i(\phi+\xi)-\rho}\Big)}dy
\end{align*}
and
\[
S_2(y):=\sum_{\substack{|1+i(\phi+\xi)-\rho|\leq K/\log X \\ |1+y+i(\phi+\xi)-\rho|\geq 1/\log X}}-\Re{\Big(\frac{\hat{u}(1-(1+y+i(\phi+\xi)-\rho)\log X)}{1+y+i(\phi+\xi)-\rho}\Big)}dy,
\]
we restate \eqref{eqn:zerosum_F2} as
\begin{align}
\label{eq:S12def}
\frac{\lambda y_0}{16}\le \int_{-\lambda}^{\lambda}S_1(y)dy+\int_{-\lambda}^{\lambda}S_2(y)dy.
\end{align}

First, consider $S_1(y)$. If $|1+y+i(\phi+\xi)-\rho|\leq 1/\log X$, then it follows from the expansion $\hat{u}(1-(1+y+i(\phi+\xi)-\rho)\log X)=1+O(|1+y+i(\phi+\xi)-\rho|\log X)$ that
\begin{align*}
S_1(y)=-\sum_{\substack{|1+i(\phi+\xi)-\rho|\leq K/\log X \\ |1+y+i(\phi+\xi)-\rho|< 1/\log X}}\frac{1+y-\beta}{|1+y+i(\phi+\xi)-\rho|^2}+O\Big(\log X\sum_{\substack{|1+i(\phi+\xi)-\rho|\leq K/\log X \\ |1+y+i(\phi+\xi)-\rho|<1/\log X}}1\Big).
\end{align*}
Using the change of variables $u = 1+y-\beta$, we now calculate
\begin{align*}
\nonumber \int_{-\lambda}^{\lambda}S_1(y)dy&=\sum_{\substack{\rho \\ |1+i(\phi+\xi)-\rho|\leq K/\log X}}\int_{\substack{|1+y+i(\phi+\xi)-\rho|<1/\log X \\ -\lambda\leq y \leq \lambda}}-\frac{1+y-\beta}{|1+y+i(\phi+\xi)-\rho|^2}dy \\
&+O\Big(\lambda(\log X)\sup_{-\lambda\leq y\leq \lambda}\sum_{\substack{|1+i(\phi+\xi)-\rho|\leq K/\log X \\ |1+y+i(\phi+\xi)-\rho|<1/\log X}}1 \Big)\\
&=\sum_{\substack{\rho \\ |1+i(\phi+\xi)-\rho|\leq K/\log X}}\int_{\substack{u^2+(\phi+\xi-\gamma)^2<1/(\log X)^2 \\ -\lambda+1-\beta\leq u \leq \lambda+1-\beta}}-\frac{u}{u^2+(\phi+\xi-\gamma)^2}du \\
&+O\Big(\#\Big \{ \rho\colon |1+i(\phi+\xi)-\rho|\leq \frac{K}{\log X}\Big\}\Big),
\end{align*}
where the last line uses \eqref{eq:Xconds}.  For the $u$-integral, there exist $a,b\in\R$ with $a<b$ such that
\begin{equation}
\label{eqn:int_prep}
\int_{\substack{u^2+(\phi+\xi-\gamma)^2<1/(\log X)^2 \\ -\lambda+1-\beta\leq u \leq \lambda+1-\beta}}-\frac{u}{u^2+(\phi+\xi-\gamma)^2}du=\int_{a}^{b}-\frac{u}{u^2+(\phi+\xi-\gamma)^2}du.
\end{equation}
Since $0\leq\beta\leq 1$, our range of $\lambda$ in \eqref{eq:A-y} ensures that $|\lambda-1+\beta|\leq \lambda+1-\beta$.  This implies that $|a|<b$.  If $a\geq 0$, then the integrand is nonpositive on $[a,b]$ and the $u$-integral is nonpositive.  If $a<0$, then we use the fact that the integrand is an odd function of $u$ to conclude that if $0<\epsilon<|a|$, then \eqref{eqn:int_prep} equals
\[
\int_{[a,b]-(-\epsilon,\epsilon)}-\frac{u}{u^2+(\phi+\xi-\gamma)^2}du=\frac{1}{2}\log\frac{a^2+(\phi+\xi-\gamma)^2}{b^2+(\phi+\xi-\gamma)^2}\leq 0.
\]
We conclude that
\begin{align}
\label{eq:s1b}
\int_{-\lambda}^{\lambda}S_1(y)dy\ll \#\Big \{ \rho\colon |1+i(\phi+\xi)-\rho|\leq \frac{K}{\log X}\Big\}.
\end{align}

Now, consider $S_2(y)$. We have that 
\begin{align*}
S_2(y)\ll (\log X)\sum_{\substack{|1+i(\phi+\xi)-\rho|\leq K/\log X}}|\hat{u}(1-(1+y+i(\phi+\xi)-\rho)\log X)|,
\end{align*}
hence
\begin{align*}
\int_{-\lambda}^{\lambda}S_2(y)\ll \lambda(\log{X})\sum_{\substack{|1+i(\phi+\xi)-\rho|\leq K/\log X}}\sup_{-\lambda\leq y\leq \lambda}|\hat{u}(1-(1+y+i(\phi+\xi)-\rho)\log X)|.
\end{align*}
Recall that $u$ has support contained in $[1,e]$.  Applying Lemma~\ref{lem:u1} and using~\eqref{eq:Xconds}, we find that if $|y|\le \lambda$, then
\begin{align*}
|\hat{u}(1-(1+y+i(\phi+\xi)-\rho)\log X)|\le |\hat{u}(1-(1+y-\beta)\log X)| \le \int_{1}^{e}\frac{u(t)}{t^{y\log{X}}}dt\ll 1.
\end{align*}
It now follows from \eqref{eq:Xconds} that
\begin{equation}
\label{eq:s2b}
\int_{-\lambda}^{\lambda}S_2(y)dy\ll  \#\Big\{ \rho\colon |1+i(\phi+\xi)-\rho|\leq \frac{K}{\log X}\Big\}.
\end{equation}

Combining \eqref{eq:S12def}, \eqref{eq:s1b}, and \eqref{eq:s2b}, we conclude that
\begin{align*}
 \#\Big\{ \rho\colon |1+i(\phi+\xi)-\rho|\leq \frac{K}{\log X}\Big\}\gg \lambda y_0.
\end{align*}
Therefore, by~\eqref{eq:Xconds}, \eqref{eqn:Kmain1+lamdalb}, and \eqref{eq:11}, we have obtained
\begin{align*}
\#\Big \{ \rho\colon |1+i(\phi+\xi)-\rho|\leq \frac{e^4}{\Cr{small}}\lambda \Big(\frac{e^8\Cr{B_const}}{\Cr{small}}\cdot \frac{ \kappa \log{\CC}}{y_0}\Big)^{\delta} \Big\}\gg \lambda y_0.
\end{align*}
Since $\Cr{small}\in(0,1)$, $\Cr{B_const}\geq 1$, and $\delta\in(0,\frac{1}{2}]$, we conclude (via the bound for $\xi$ in \eqref{eq:sss1-9} and the triangle inequality) the weaker estimate
\begin{equation}
\label{eqn:main_inequality}
\#\Big \{ \rho\colon |1+i\phi-\rho|\leq \frac{2e^4}{\Cr{small}}\lambda \Big(\frac{e^8\Cr{B_const}}{\Cr{small}}\cdot \frac{ \kappa \log{\CC}}{y_0}\Big)^{\delta} \Big\}\gg \lambda y_0.
\end{equation}

\begin{proof}[Proof of \cref{thm:main}]
Let $y_0$ satisfy \eqref{eqn:ranges_3.3} and $\lambda$ satisfy \eqref{eq:A-y}.  In the general case, we deduce \cref{thm:main} from \eqref{eqn:main_inequality} by choosing
\[
x=e^{y_0},\qquad L = \frac{4e^4}{\Cr{small}}\Big(\frac{e^8\Cr{B_const}\kappa}{\Cr{small}}\Big)^{\delta}\lambda y_0.
\]
When the numbers $a_{\pi}(n)$ are real, \cref{lem:twist} states that $|\phi|\leq \Cr{phi_real} y_0^{-1}N^{6/\kappa}$.  If we impose the additional restriction that
\[
\lambda\geq \frac{\Cr{small}\Cr{phi_real}}{2e^4}\cdot \frac{ N^{6/\kappa}}{y_0}\Big(\frac{e^8\Cr{B_const}}{\Cr{small}}\cdot \frac{\kappa \log{\CC}}{y_0}\Big)^{-\delta},
\]
then we can weaken \eqref{eqn:main_inequality} to
\begin{equation}
\label{eqn:main_inequality2}
\#\Big \{ \rho\colon |1-\rho|\leq \frac{4e^4}{\Cr{small}}\lambda \Big(\frac{e^8\Cr{B_const}}{\Cr{small}}\cdot \frac{ \kappa \log{\CC}}{y_0}\Big)^{\delta} \Big\}\gg \lambda y_0.
\end{equation}
We now choose
\[
x=e^{y_0},\qquad L = \frac{8e^4}{\Cr{small}}\Big(\frac{e^8\Cr{B_const}\kappa}{\Cr{small}}\Big)^{\delta}\lambda y_0.\qedhere
\]
\end{proof}
\begin{proof}[Proof of \cref{cor:main}]
Apply \cref{thm:main} with $m=\kappa=1$ and $A_0\ll 1$.  Also, there exists an absolute constant $\Cl[abcon]{N_bound}>0$ such that if $\chi$ has order $k\geq 2$, then $|\phi|\leq y_0^{-1}(\Cr{N_bound}N)^{2k^2}$ (see \cite[Lemma 3.3]{GS_JEMS}).
\end{proof}
\begin{proof}[Proof of \cref{cor:main_2}]
This follows from \cref{thm:main} with $m=\kappa=2$ and $A_0\ll1$.  For these cusp forms, the Fourier coefficients $\lambda_f(n)$ (hence $a_f(n)$) are real.
\end{proof}

\bibliographystyle{abbrv}
\bibliography{KerrKlurmanThorner_large_sums.bib}
\end{document}